\documentclass{amsart}

\usepackage{amsmath, amsthm, amsfonts}
\usepackage{amssymb}
\usepackage[utf8]{inputenc}
\usepackage{latexsym}
\usepackage{verbatim}
\usepackage{url}
\usepackage{textcomp}
\usepackage[all,cmtip]{xy}
\usepackage{color}
\usepackage{hyperref}
\definecolor{mylinkcolor}{rgb}{0.0,0.0,0.65}
\definecolor{mycitecolor}{rgb}{0.0,0.0,0.65}
\definecolor{myurlcolor}{rgb}{0.0,0.0,0.65}
\hypersetup{colorlinks=true,urlcolor=myurlcolor,citecolor=mycitecolor,linkcolor=mylinkcolor,linktoc=page,breaklinks=true}

\usepackage{tabularx}
\input{xypic}
\usepackage{enumerate}
\usepackage{amssymb}
\usepackage{array}

\DeclareMathAlphabet{\mathfr}{U}{euf}{m}{n}

\newtheorem{theorem}{Theorem}[section]
\newtheorem*{theorem*}{Theorem}

\newtheorem{proposition}[theorem]{Proposition}
\newtheorem{corollary}[theorem]{Corollary}

\newtheorem{lemma}[theorem]{Lemma}

\theoremstyle{remark}
\newtheorem{remark}[theorem]{Remark}

\newcommand{\Q}{\mathbb Q}
\newcommand{\Qbar}{{\overline{\mathbb Q}}}

\newcommand{\Gal}{\mathrm{Gal}}
\newcommand{\R}{\mathbb R}
\newcommand{\Z}{\mathbb Z}

\newcommand{\C}{\mathbb C}
\newcommand{\F}{\mathbb F}

\newcommand{\triv}{\mathbf{1}}
\newcommand{\GL}{\mathrm{GL}}

\newcommand{\SM}{\operatorname{SM}}
\newcommand{\End}{\operatorname{End}}
\newcommand{\Hom}{\operatorname{Hom}}
\newcommand{\Frob}{\operatorname{Fr}}
\newcommand{\Ind}{\operatorname{Ind}}

\newcommand{\USp}{\operatorname{USp}}
\newcommand{\ad}{\operatorname{ad}}
\newcommand{\thetabar}{{\overline \theta}}
\newcommand{\chibar}{{\overline \chi}}
\newcommand{\psibar}{{\overline \psi}}
\newcommand{\Aut}{\operatorname{Aut}}

\newcommand{\p}{\mathfrak{p}}

\newcommand{\Tr}{\operatorname{Tr}}

\newcommand{\ST}{\operatorname{ST}}
\newcommand{\id}{\operatorname{id}}
\newcommand{\Unitary}{\operatorname{U}}
\newcommand{\Nm}{\operatorname{Nm}}

\newcommand{\cA}{\mathcal{A}}

\newcommand{\ba}{\mathbf{a}}
\newcommand{\be}{\mathbf{e}}
\newcommand{\Conj}{\mathrm{Conj}}

\numberwithin{equation}{section}

\usepackage{url}
\newcommand{\Ab}{\textbf{A}}
\newcommand{\Bb}{\textbf{B}}
\newcommand{\Cb}{\textbf{C}}
\newcommand{\Db}{\textbf{D}}
\newcommand{\Eb}{\textbf{E}}
\newcommand{\Fb}{\textbf{F}}
\newcommand{\Gb}{\textbf{G}}
\newcommand{\Hb}{\textbf{H}}
\newcommand{\Ib}{\textbf{I}}
\newcommand{\Jb}{\textbf{J}}
\newcommand{\Kb}{\textbf{K}}
\newcommand{\Lb}{\textbf{L}}
\newcommand{\Mb}{\textbf{M}}
\newcommand{\Nb}{\textbf{N}}

\newcommand{\cyc}[1]{{\mathrm{C}_#1}}
\newcommand{\dih}[1]{{\mathrm{D}_#1}}

\newcommand{\sym}[1]{{\mathrm{S}_#1}}

\newcommand*{\stgroup}[2][]{\href{https://www.lmfdb.org/SatoTateGroup/#2}{{\ifx&#1& #2 \else #1 \fi}}}

\newcommand{\arXiv}[2]{\href{https://arxiv.org/abs/#1}{arXiv:#1v#2}}

\begin{document}
\title[]{On a local-global principle\\ for quadratic twists of abelian varieties}
\author{Francesc Fit\'e}

\address{Departament de matem\`atiques i inform\`atica and Centre de recerca matem\`atica,
Universitat de Barcelona,
Gran via de les Corts Catalanes 585, 08007 Barcelona, Catalonia}
\email{ffite@ub.edu}
\urladdr{http://www.ub.edu/nt/ffite/}

\begin{abstract} 
Let $A$ and $A'$ be abelian varieties defined over a number field~$k$ of dimension $g\geq 1$. For $g\leq 3$, we show that the following local-global principle holds: $A$ and $A'$ are quadratic twists of each other if and only if, for almost all primes $\p$ of $k$ of good reduction for $A$ and $A'$, the reductions $A_\p$ and~$A_\p'$ are quadratic twists of each other. This result is known when $g=1$, in which case it has appeared in works by Kings, Rajan, Ramakrishnan, and Serre. We provide an example that violates this local-global principle in dimension $g=4$.
\end{abstract}

\maketitle

\section{Introduction}\label{section: introduction} 

\subsection{Main result} Let $k$ be a number field and let $A$ be an abelian variety defined over $k$ of dimension $g\geq 1$. We will denote by~$A_\psi$ the twist of $A$ corresponding to a given $1$-cocycle $\psi\colon G_k\rightarrow \Aut(A_{\Qbar})$, where $G_k$ denotes the absolute Galois group of $k$.
We will say that an abelian variety $A'$ defined over $k$ is a \emph{quadratic twist of~$A$} if it is isogenous over $k$ to $A_\chi$ for some (possibly trivial) quadratic character $\chi \colon G_k \rightarrow \{\pm 1\}\subseteq \Aut(A_\Qbar)$. 
Let $\Sigma_k$ be the set of nonzero prime ideals of the ring of integers of~$k$. For~$\p\in\Sigma_k$, denote by $k_\p$ the residue field of $k$ at $\p$, and if $\p$ is of good reduction for $A$, let~$A_\p$ be the reduction of $A$ modulo~$\p$. Let $f_{A_\p}$ be the  the characteristic polynomial of Frobenius relative to $k_\p$ acting on an $\ell$-adic Tate module of $A_\p$ for some prime $\ell$ coprime to $\p$. If $\p$ is of good reduction for both $A$ and $A'$, we will say that~$A$ and~$A'$ are \emph{locally quadratic twists at $\p$} if $A_\p$ and $A'_\p$ are quadratic twists, that is, if either
$$
f_{A_\p}(T)=f_{A'_\p}(T)\qquad \text{or}\qquad f_{A_\p}(T)=f_{A'_\p}(-T)\,.
$$  
One easily sees that if $A$ and $A'$ are quadratic twists, then 
\begin{equation}\label{equation: main condition}
\text{$A$ and $A'$ are locally quadratic twists at almost all primes of $\Sigma_k$.}
\end{equation}

Throughout the article, by ``almost all primes of $\Sigma_k$" we mean ``all primes in a (Dirichlet) density~$1$ subset of $\Sigma_k$". Our main object of study will be the converse of the above implication. The following is our main result.

\begin{theorem}\label{theorem: Main}
Let $A$ and $A'$ be abelian varieties defined over $k$ of dimension $g \leq 3$. Suppose that \eqref{equation: main condition} holds for $A$ and $A'$. Then $A$ and $A'$ are quadratic twists. 
\end{theorem}

We complement the above theorem by providing two abelian fourfolds defined over $\Q$ that are locally quadratic twists at all odd primes, but which are not quadratic twists. Let $a_\p(A)$ denote the Frobenius trace of $A$ at $\p$. The above theorem is not true if one replaces condition \eqref{equation: main condition} by the weaker condition that $a_\p(A)$ and $a_\p(A')$ coincide up to sign for almost all $\p$ in $\Sigma_k$.
Under this weaker hypothesis, it is easy to find counterexamples, and we give one in dimension $g=2$. Both examples can be found in \S\ref{section: examples}. 

\subsection{Previous results} 
Theorem \ref{theorem: Main} is known if $g=1$ (at least if~$A$ does not have complex multiplication \emph{defined over}~$k$). There are different proofs of this fact in works of Serre \cite[p. 324]{Ser72}, Kings \cite[p. 90-91]{Kin98}, and Ramakrishnan \cite[Thm. B]{Ram00} (note that in these works the results are generally phrased in terms of Galois representations attached to modular forms). In the CM case, there is a closely related result by Wong \cite[Cor. 1]{Won99}. 

The dimension $g=1$ case of Theorem \ref{theorem: Main} can also be retrieved from a general result on $\ell$-adic representations due to Rajan \cite{Raj98}. More in general, as we will explain in \S\ref{section: generalities}, this result implies that if two abelian varieties $A$ and $A'$ satisfy~\eqref{equation: main condition}, then $A_\Qbar$ and $A'_\Qbar$ are isogenous. Refining this conclusion to the statement of Theorem~\ref{theorem: Main} is the goal of this article. 

The above mentioned results are just a few instances of a vast literature on local-global principles concerning abelian varieties over number fields. Our problem is closely related to a recent result of Khare and Larsen \cite{KL20}. They show that the base changes $A_\Qbar$ and $A'_\Qbar$ are isogenous if and only if for almost\footnote{Actually in \cite{KL20} it suffices to assume that a sufficiently large density of primes of~$\Sigma_k$ has the required property. Obtaining an analogous refinement of Theorem \ref{theorem: Main} seems an interesting question, which we have not attempted to address in this article.} all primes $\p$ in~$\Sigma_k$ so are the base changes $A_\p\times \overline k_\p$ and $A'_\p\times \overline k_\p$, where $\overline k_\p$ denotes an algebraic closure of $k_\p$. Our methods of proof and those of \cite{KL20} are rather different (mainly due to the fact that ours is a question sensitive to base change). 

In fact, we only prove the local-global principle investigated in this article for certain classes of abelian varietes. The fact that these classes end up encompassing those of dimension $g\leq 3$ may be regarded as an ``accident in low dimension". This is reminiscent of other local-global questions concerning abelian varieties over number fields. For example, Katz \cite{Kat81} showed that if $g\leq 2$ and for almost all $\p$ the cardinality of $A_\p(k_\p)$ is divisible by some prime number~$m$, then there exists an abelian variety $A'$ isogenous to~$A$ such that $m$ divides the order of the torsion subgroup of $A'(k)$; he also showed that this fails to be true if $g\geq 3$. 

\subsection{Outline of the article} The methods employed to prove Theorem \ref{theorem: Main} all involve the study of the $\ell$-adic representation $\varrho_{A,\ell}$ attached to the abelian variety~$A$. Let us briefly summarize them. We will say that a subset $\cA$ of the set of abelian varieties defined over the number field~$k$ satisfies the \emph{local-global QT principle} if any pair of abelian varieties in~$\cA$ that satisfy \eqref{equation: main condition} are quadratic twists.
 
In \S\ref{section: generalities}, we record some background results. We start by deriving some consequences of Faltings isogeny theorem; this implies, for example, that if $A$ and $A'$ satisfy \eqref{equation: main condition}, then $A$ and $A'$ share the same endomorphism field $K$. We then show that the result by Rajan mentioned above implies that the local-global QT principle holds for those abelian varieties~$A$ such that $\End(A_\Qbar)=\Z$. We conclude \S\ref{section: generalities} by describing some connections with the theory of Sato--Tate groups. Our proof of Theorem \ref{theorem: Main} is independent of the Sato--Tate conjecture, but it does benefit from the classification of Sato--Tate groups of abelian varieties of dimension $\leq 3$.

Suppose that $A$ and $A'$ satisfy \eqref{equation: main condition}. The main results of \S\ref{section: Rajan} are two variants of Rajan's theorem that build on \cite{Ser81}: Theorem~\ref{theorem: serre} applies when $\varrho_{A,\ell}$ has an unrepeated strongly absolutely irreducible factor; Theorem \ref{theorem: mild refinement serre} shows that, under a certain technical assumption, the base changes~$A_K$ and~$A'_K$ are quadratic twists. A common strategy in later sections consists on first applying Theorem \ref{theorem: mild refinement serre}, and then using arguments specific to the situation of interest to descend the validity of the local-global QT principle from $K$ to $k$.

In \S\ref{section: productsec}, we consider families of abelian varieties for which the methods described in the previous paragraph fail to apply. Those include abelian varieties that are geometrically isogenous to the  power of an elliptic curve. In this case, the Tate module tensor decompositions obtained in \cite{FG20} allow to translate our problem concerning $\ell$-adic representations of degree $2g$ into one concerning Artin representations of degree $g$. This reduction by half in the degree of the representations turns out to be crucial. Indeed, the corresponding local-global principle for Artin representations is almost immediate when $g$ is odd, and, while it can fail for even $g$, it does hold for $g=2$. This can be seen by means of reinterpreting Ramakrishnan's theorem \cite[Thm. B]{Ram00} in the context of Artin representations. Hence, Ramakrishnan's theorem, originally conceived to treat the dimension $g=1$ case of our problem, ends up playing an important role in its generalization in dimension $g=2$.
In \S\ref{section: productsecCM}, we show that produtcs of geometrically pairwise nonisogenous elliptic curves with complex multiplication satisfy the local-global QT principle. In this section, we use Hecke's equidistribution theorem \cite{Hec20} at several points; this is natural and intuitive, but it would have sufficed to use instead results from \cite{Ser81}.

In \S\ref{section: proof}, we complete the proof of Theorem \ref{theorem: Main}. The proof takes into account the different possibilities for $\End(A_\Qbar)\otimes \Q$. Many cases were covered in previous sections. Other cases, as mentioned above, involve invoking Theorem \ref{theorem: serre} or Theorem~\ref{theorem: mild refinement serre}, and then providing a few adhoc arguments. As an example of the latter, we mention the case in which $A$ is isogenous to the product of an elliptic curve $E$ and and abelian surface $B$ such that $\Hom(E_\Qbar,B_\Qbar)=0$. If $E$ does not have CM, then Theorem~\ref{theorem: Main} immediately follows from Theorem~\ref{theorem: serre}. However, if $E$ has CM, Theorem \ref{theorem: serre} essentially only provides two quadratic characters $\chi$ and $\psi$ such that $A'\sim E_\chi \times B_\psi$.  
That one can take $\chi=\psi$ is the content of Proposition~\ref{proposition: keycaseL}. 

\subsection{Notation and terminology} Throughout the article, $k$ is a number field, and~$A$ and~$A'$ are abelian varieties defined over $k$ of dimension $g\geq 1$. All algebraic extensions of $k$ are assumed to be contained in a fixed algebraic closure $\Qbar$ of $\Q$. If $L/k$ is one such field extension, then we will write $G_L$ to denote the absolute Galois group $\Gal(\Qbar/L)$, and $A_L$ to denote the base change $A\times_k L$. We denote by $\End^0(A)$ the endomorphism algebra of $A$, and refer to $\End^0(A_\Qbar)$ as the geometric endomorphism algebra of $A$. However, due to a widely used convention, when we say that $A$ has complex multiplication (CM), we in fact mean that $A_\Qbar$ has CM. Whenever it becomes necessary to ask that the CM be defined over $k$, we will explicitly specify it. If $E$ is a field and $\theta\colon G_k\rightarrow \GL_r(E)$ is a representation, then $\theta|_L$ denotes the restriction of~$\theta$ to~$G_L$, and $\theta^\vee$ denotes the contragredient representation of $\theta$. We will say that $\theta$ is \emph{absolutely irreducible} if $\theta\otimes \overline E$ is irreducible, where $\overline E$ denotes an algebraic closure of $E$. We will say that $\theta$ is \emph{strongly absolutely irreducible} if $\theta|_L$ is absolutely irreducible for every finite extension $L/k$.

\subsection{Acknowlegements.} Thanks to Edgar Costa for his assistance in \S\ref{section: examples}, to Xavier Guitart for discussions that led to this investigation, to Kiran Kedlaya for discussions around Lemma \ref{lemma: same ST groups under ST}, and to Bjorn Poonen for explaining to me the proof of Lemma~\ref{lemma: splittingoverk}. I was financially supported by the Simons Foundation grant 550033, the Ram\'on y Cajal fellowship RYC-2019-027378-I, and the Mar\'ia de Maeztu Program CEX2020-001084-M.

\section{Generalities}\label{section: generalities}

We will derive some properties that~$A$ and~$A'$ must satisfy in case \eqref{equation: main condition} holds. First we record some elementary consequences of Faltings isogeny theorem \cite{Fal83}; then we describe some implications of a theorem of Rajan; finally we explain the connection of our problem to the theory of Sato--Tate groups and derive some consequences of their classification for abelian surfaces.

\subsection{Consequences of Faltings isogeny theorem} For a prime number~$\ell$, let $V_\ell(A)$ denote the rational $\ell$-adic Tate module of~$A$. There is a continuous action of $G_k$ on $V_\ell(A)$ which gives rise to an $\ell$-adic representation $\varrho_{A,\ell}\colon G_k\rightarrow \Aut(V_\ell(A))$. For $\p \in \Sigma_k$, denote by $\Frob_\p$ an arithmetic Frobenius element at $\p$. By the work of Weil, if $\p$ is of good reduction for $A$ and does not divide $\ell$, then 
$$
L_\p(A,T):=\det(1-\varrho_{A,\ell}(\Frob_\p)T)
$$
is a polynomial of degree $2g$ with integer coefficients which does not depend on the choice of $\ell$. 
One easily verifies that $\varrho_{A_\chi,\ell}\simeq \chi \otimes \varrho_{A,\ell}$ for any quadratic character~$\chi$, and Faltings isogeny theorem then gives that
$A'$ is a quadratic twist of $A$ by~$\chi$ if and only if $\varrho_{A',\ell}\simeq \chi\otimes \varrho_{A,\ell}$. Together with the fact that $L_\p(A,T)$ coincides with the reverse Weil polynomial $T^{2g}f_{A_\p}(1/T)$ of the reduction~$A_\p$, this has the following consequence.

\begin{lemma}\label{lemma: globaltolocal}
If $A$ and $A'$ are quadratic twists, then they are locally quadratic twists at every prime~$\p$ which is of good reduction for both.
\end{lemma}

Throughout this article we will denote by $K/k$ the \emph{endomorphism field of $A$}, that is, the minimal extension $K/k$ such that $\End(A_K)\simeq \End(A_\Qbar)$. It is a finite and Galois extension.

\begin{lemma}\label{lemma: endomorphismfield}
Suppose that \eqref{equation: main condition} holds for $A$ and $A'$. Then $\End^0(A_L)$ and $\End^0(A'_L)$ have the same dimension for every finite extension $L/k$.
In particular, $A$ and $A'$ share the same endomorphism field.
\end{lemma}

\begin{proof}
The lemma follows from the isomorphism
$$
\End(A_L)\otimes_\Z \Q_\ell\simeq (V_\ell(A)\otimes V_\ell(A)^\vee)^{G_L}\simeq (V_\ell(A')\otimes V_\ell(A')^\vee)^{G_L}\simeq \End(A'_L)\otimes_\Z\Q_\ell\,.
$$
The central isomorphism follows from \eqref{equation: main condition} and the Chebotarev density theorem.
\end{proof}

\subsection{Rajan's theorem}\label{section: Rajan first approach}

Let $E_\lambda$ denote a finite extension of $\Q_\ell$, and let $\varrho,\varrho'\colon G_k \rightarrow \GL_r(E_\lambda)$ be continuous, semisimple representations unramified outside a finite set $S\subseteq \Sigma_k$.  Consider the set
$$
\SM(\varrho,\varrho'):=\{\p \in \Sigma_k-S\,|\, \Tr(\varrho(\Frob_\p))=\Tr(\varrho'(\Frob_\p)) \}\,.
$$
With the above notations, Rajan \cite[Thm. 2, part i)]{Raj98} shows the following.

\begin{theorem}[Rajan]\label{theorem: Rajan} Let $F/k$ be a finite extension such that the Zariski closures of $\varrho(G_F)$ and $\varrho'(G_F)$ are connected. If the upper density of $\SM(\varrho|_F,\varrho'|_F)$ is positive, then there exists a finite Galois extension $L/k$ containing $F/k$ such that $\varrho|_L\simeq \varrho'|_L$.
\end{theorem}

Rajan's theorem relates to the local-global QT principle by means of next lemma.

\begin{lemma}\label{lemma: plusminusbasechange}
Suppose that there exists a density 1 subset $\Sigma \subseteq \Sigma_k$ such that for every $\p \in \Sigma$ there exists $\epsilon_\p\in \{\pm 1\}$ such that $\Tr(\varrho'(\Frob_\p))= \epsilon_\p\cdot \Tr(\varrho(\Frob_\p))$. Then for every finite extension $F/k$, the set $\SM(\varrho|_F,\varrho'|_F)$ has a positive upper density.
\end{lemma}

\begin{proof}
We first claim that there exists a density 1 subset $\Sigma' \subseteq \Sigma_F$ such that for every $\p \in \Sigma'$ there exists $\chi_\p\in \{\pm 1\}$ such that $\Tr(\varrho'(\Frob_\p))= \chi_\p\cdot \Tr(\varrho(\Frob_\p))$. 
Indeed, the hypothesis implies that $\varrho\otimes \varrho \simeq \varrho'\otimes \varrho'$. In particular, we have $\varrho\otimes \varrho|_F \simeq \varrho'\otimes \varrho'|_F$. Hence the claim.  
But this implies that, if the lemma were false, then $\Tr(\varrho(\Frob_\p))=-\Tr(\varrho'(\Frob_\p))$ for all $\p$ in a density $1$ subset of $\Sigma_F$. The Chebotarev density theorem would then imply that $\Tr(\varrho|_F)=-\Tr(\varrho'|_F)$, which contradicts $\Tr(\varrho(\id))=\Tr(\varrho'(\id))$.
\end{proof}

Rajan's theorem has the following consequence (see \cite[Thm. 2, part iii)]{Raj98}).

\begin{corollary}[Rajan]\label{corollary: Rajan}
Suppose that the hypotheses of Theorem \ref{theorem: Rajan} hold, and that moreover~$\varrho$ is strongly absolutely irreducible. Then there exists a finite Galois extension $L/k$ and a character $\chi \colon \Gal(L/k)\rightarrow \overline \Q^\times$ such that $\varrho'\simeq \chi\otimes \varrho$.
\end{corollary}

\begin{proof}
By Theorem \ref{theorem: Rajan} there exists a finite Galois extension $L/k$ such that $\varrho|_L\simeq \varrho'|_L$. Since $\varrho$ is strongly absolutely irreducible, Schur's lemma shows that the space
$$
\Hom_{G_L}(\varrho',\varrho)\simeq (\varrho'\otimes \varrho^\vee)^ {G_L}
$$
is $1$-dimensional. Let $\chi$ denote the character of the action of $\Gal(L/k)$ on this space. Then
$$
\Hom_{G_k}(\varrho',\chi\otimes \varrho)\simeq \Hom_{G_k}(\varrho'\otimes \varrho^\vee, (\varrho'\otimes \varrho^\vee)^ {G_L})\not =0\,.
$$
Since $\chi\otimes \varrho$ is irreducible and has the same dimension as $\varrho'$, we have $ \varrho'\simeq \chi\otimes \varrho$. 
\end{proof}

We will apply Theorem \ref{theorem: Rajan} when~$\varrho$ and $\varrho'$ are the $\ell$-adic representations attached to $A$ and $A'$. The following is a consequence of Rajan's theorem.

\begin{corollary}\label{corollary: locquad geomiso}
Suppose that \eqref{equation: main condition} holds for $A$ and $A'$. Then there exists a finite Galois extension $L/k$ such that $A_L$ and $A'_L$ are isogenous. 
\end{corollary}

\begin{proof}
There exists a finite extension $F/k$ such that the Zariski closures of $\varrho_{A,\ell}(G_F)$ and $\varrho_{A',\ell}(G_F)$ are connected. By Lemma~\ref{lemma: plusminusbasechange}, the representations $\varrho_{A,\ell}|_F$ and $\varrho_{A',\ell}|_F$ satisfy the hypotheses of Theorem~\ref{theorem: Rajan}. Hence, there exists a finite Galois extension $L/k$ containing $F/k$ such that $\varrho_{A,\ell}|_L\simeq \varrho_{A',\ell}|_L$, and so $A_L$ and $A'_L$ are isogenous.
\end{proof}

We now use the above corollaries to show that the local-global QT principle holds for abelian varieties with trivial geometric endomorphism ring.

\begin{corollary}\label{corollary: trivend}
Suppose that $\End(A_\Qbar)\simeq \Z$ and that \eqref{equation: main condition} holds for $A$ and $A'$. Then $A$ and $A'$ are quadratic twists. 
\end{corollary}

\begin{proof}
By Faltings isogeny theorem, for every finite extension $L/k$ we have an isomorphism
$$
\End_{G_L}(V_\ell(A))\otimes_{\Q_\ell} \Qbar_\ell \simeq \End(A_L)\otimes_{\Z}\Qbar_\ell\simeq \Qbar_\ell\,,
$$
which implies that $\varrho_{A,\ell}$ is strongly absolutely irreducible. By Corollary \ref{corollary: Rajan}, there exists a finite Galois extension $L/k$ and a character $\chi$ of $\Gal(L/k)$ such that $\varrho_{A',\ell}\simeq \chi \otimes \varrho_{A,\ell}$. Again by Faltings isogeny theorem, there is an isomorphism 
$$
 \Hom_{G_L}(V_\ell(A),V_\ell(A')) \simeq \Hom(A_L,A'_L)\otimes_{\Z}\Q_\ell\simeq\Q_\ell(\chi)\,,
$$
where $\Q_\ell(\chi)$ means $\Q_\ell$ equipped with the action of $\Gal(L/k)$ via $\chi$. Hence $\Gal(L/k)$ acts on $\Hom(A_L,A'_L)\simeq \Z$ via $\chi$, and thus $\chi$ must be quadratic.  
\end{proof}

\subsection{The connection with Sato--Tate groups}

Throughout this section suppose that $A$ and $A'$ have dimension $g \leq 3$. The Sato--Tate group of $A$, denoted $\ST(A)$, is a closed real Lie subgroup of $\USp(2g)$, only defined up to conjugacy. It captures important arithmetic information of $A$ and it is conjectured to predict the limiting distribution of the Frobenius elements attached to~$A$. See \cite{BK15} for its definition in our context; see \cite{Ser12} for a conditional definition in a more general context.
Let us denote by $\ST(A)^0$ the connected component of the identity in $\ST(A)$, and by $\pi_0(\ST(A))$ the group of components of $\ST(A)$.

\begin{lemma}\label{lemma: ST group}
Suppose that \eqref{equation: main condition} holds for $A$ and $A'$. Then 
$$
\pi_0(\ST(A))\simeq \pi_0(\ST(A'))\qquad \text{and}\qquad \ST(A)^0\simeq \ST(A')^0\,.
$$
\end{lemma}

\begin{proof}
By Lemma \ref{lemma: endomorphismfield},~$A$ and~$A'$ have the same endomorphism field $K/k$, and by \cite[Prop. 2.17]{FKRS12}, there is an isomorphism $\pi_0(\ST(A))\simeq \Gal(K/k)$. The second isomorphism follows from Corollary \ref{corollary: Rajan} and \cite[Rem. 3.2]{BK15}.
\end{proof}

For every $\p\in\Sigma_k$ of good reduction for $A$, one defines a semisimple conjugacy class~$s_\p$ of $\ST(A)$ as in \cite[Def. 2.9]{FKRS12}. The Sato--Tate conjecture for~$A$ is the prediction that the sequence $\{s_\p\}_\p$, where the indexing set of primes is ordered by norm, is equidistributed with respect to the projection of the Haar measure of $\ST(A)$ on its set of conjugacy classes. The characteristic polynomial $\sum_ia_iT^i\in \R[T]$ of an element of $\USp(2g)$ is monic and palindromic, and it is hence determined by the $g$-tuple $\ba=(a_1,\dots,a_g)$. If we denote by $X_g$ the set of $g$-tuples obtained in this way, then there is a bijection $\Conj(\USp(2g))\simeq X_g$. Let $\ba_\p$ denote the image of $s_\p$ under the map
$$
\Conj(\ST(A))\rightarrow \Conj(\USp(2g))\simeq X_g\,.
$$
We will denote by $\mu$ the projection on $X_g$, via the above map, of the Haar measure of $\ST(A)$. Define similarly $s'_\p$, $\ba_\p'$ and $\mu'$. 

Let $\pi(x)$ denote the number of primes $\p$ in $\Sigma_k$ of good reduction for $A$ such that $\Nm(\p)\leq x$. The Sato--Tate conjecture implies that 
\begin{equation}\label{equation: ST conjecture}
\int_{\ba \in X_g}f(\ba)\mu =\lim_{x\rightarrow \infty}\frac{1}{\pi(x)}\sum_{\Nm(\p)\leq x}f(\ba_\p)
\end{equation}
for every $\C$-valued continuous function $f$ on $X_g$.

\begin{lemma}\label{lemma: same ST groups under ST}
Suppose that \eqref{equation: main condition} holds for the abelian varieties $A$ and $A'$ of dimension $g \leq 3$. If the Sato--Tate conjecture holds for $A$ and $A'$, then the Sato--Tate groups of $A$ and $A'$ coincide.
\end{lemma}

\begin{proof}
Given $\ba \in X_g$ and a $g$-tupe of nonnegative integer numbers $\be=(e_1,\dots,e_g)$, let us write $\ba^\be$ to denote $a_1^{e_1}\cdots a_g^{e_g}$. 
By \cite[Prop. 6.16]{FKS21c} and Lemma \ref{lemma: ST group}, it suffices to show that 
\begin{equation}\label{equation: moments}
\int_{\ba \in X_g}\ba^\be \mu = \int_{\ba \in X_g}\ba^\be\mu'
\end{equation}
for every $\be$. Let $w(\be)$ denote $\sum_i ie_i$. Since $-1\in \ST(A)$, we have that both members of \eqref{equation: moments} are $0$ if $w(\be)$ is odd. Suppose from now one that $w(\be)$ is even. If $A$ and $A'$ are locally quadratic twists at $\p$, then $s'_\p=\pm s_\p$, and hence ${\ba'}_{\p}^\be=\ba_{\p}^{\be}$. Then \eqref{equation: moments} follows from \eqref{equation: ST conjecture}.
\end{proof}

\begin{remark}
Lemma~\ref{lemma: same ST groups under ST} will not be used in the sequel. Since the Sato-Tate group of an abelian variety is invariant under quadratic twist, Lemma~\ref{lemma: same ST groups under ST} can be recovered from Theorem~\ref{theorem: Main} (without assuming the Sato--Tate conjecture).
\end{remark}

\subsection{Consequences of the classification of Sato--Tate groups} 

The next two results represent partial progress toward Theorem \ref{theorem: Main} and will be used to simplify some proofs in subsequent sections. We use the notations for Sato--Tate groups introduced in \cite{FKRS12} and \cite{FKS21c}, and present in the data base \cite{LMFDB}.

\begin{proposition}\label{proposition: iftwistquadratic}
Suppose that \eqref{equation: main condition} holds for the abelian surfaces $A$ and $A'$,~and:
\begin{enumerate}[i)] 
\item The Sato--Tate group of $A$ is distinct from $\stgroup[C_{4,1}]{1.4.F.4.1a}$, $\stgroup[F_{ac}]{1.4.D.4.1a}$.
\item There exists a character $\chi$ of $G_k$ such that $\varrho_{A',\ell}\simeq \chi \otimes \varrho_{A,\ell}$. 
\end{enumerate}
Then $A$ and $A'$ are quadratic twists. 
\end{proposition}

\begin{proof}
By comparing determinants of $\varrho_{A,\ell}$ and $\varrho_{A',\ell}$ we see that $\chi^{4}=1$. It will suffice to prove that in fact $\chi^2=1$. Suppose that the order of $\chi$ were $4$. Let $\Sigma\subseteq \Sigma_k$ denote the density 1 set of primes~$\p$ of good reduction for $A$ and $A'$, of absolute residue degree $1$, and such that $A$ and $A'$ are locally quadratic twists at~$\p$. 
Define $z_{2,0}$ as the upper density of the set of primes $\p \in \Sigma$ such that the coefficient $b_\p$ of $T^2$ in $L_\p(A,T)$ is~$0$. 
For every $\p \in \Sigma$ such that $\chi(\Frob_\p)$ has order $4$, condition $ii)$ implies that $b_\p$ is zero. In particular, we see that $z_{2,0}\geq 1/2$. 
It is proven in the course of \cite[Thm. 3]{Saw16} that $z_{2,0}$ is the proportion of connected components $C$ of the Sato--Tate group of $A$ such that, for all $\gamma\in C$, the coefficient of $T^2$ in the characteristic polynomial of~$\gamma$ is $0$. Interpreted as this proportion, the quantity $z_{2,0}$ can be read from \cite[Table 8]{FKRS12} (see columns `$c$' and `$z_2$'; note that $z_{2,0}$ is the central term of $z_2$). By inspection of the table, one finds that $\stgroup[C_{4,1}]{1.4.F.4.1a}$ and $\stgroup[F_{ac}]{1.4.D.4.1a}$ are the only two Sato--Tate groups of abelian surfaces for which $z_{2,0}\geq 1/2$.
\end{proof}

The next lemma accounts for cases uncovered by the previous proposition. For its proof it is convenient to introduce the following notion. For a finite Galois extension $F/k$, $A$ will be said to have \emph{Frobenius traces concentrated in $F$} if $\Tr(\varrho_{A,\ell}(\Frob_\p))=0$ for every $\p\in \Sigma_k$ of good reduction for $A$ which does not split completely in $F/k$.

\begin{lemma}\label{lemma: Cnextensions}
Let $A$ and $A'$ be abelian surfaces defined over $k$. Let $K/k$ denote the endomorphism field of $A$. Suppose that:
\begin{enumerate}[i)]
\item The Sato--Tate group of $A$ is $\stgroup[C_{4,1}]{1.4.F.4.1a}$, $\stgroup[F_{ac}]{1.4.D.4.1a}$.
\item There is a character $\psi$ of $G_k$ such that $\psi|_K$ is quadratic and $\varrho_{A',\ell}\simeq \psi \otimes \varrho_{A,\ell}$. 
\end{enumerate} 
Then $A$ and $A'$ are quadratic twists.
\end{lemma}

\begin{proof} 
Let $E/k$ be the cyclic extension cut out by $\psi$. By hypothesis~$ii)$, we have that $[E:E\cap K]\leq 2$. Hypothesis~$i)$ implies that $A$ has Frobenius traces concentrated in $K$. If $E\subseteq K$, then this implies that 
\begin{equation}\label{equation: kisogenous}
\varrho_{A',\ell}\simeq \psi\otimes \varrho_{A,\ell}\simeq \varrho_{A,\ell}\,,
\end{equation}
and hence $A$ and $A'$ are in fact isogenous. Suppose from now on that $[E:E\cap K]= 2$. 
We may assume that~$\psi$ is quartic as otherwise there is nothing to show. Then $\Gal(EK/k)$ is a group of order 8 with two normal subgroups of order 2 yielding quotients isomorphic to~$\cyc 4$. This uniquely determines the isomorphism class of $\Gal(EK/k)$, which must be $\cyc 2 \times \cyc 4$. In this case, one readily sees that there exists a quadratic subextension $F/k$ of $EK/k$ such that $EK=FK$. 
Let $\chi$ denote the nontrivial character of $\Gal(F/k)$. If follows that $\chi|_K=\psi|_K$ and hence $\psi \otimes \varrho_{A,\ell}\simeq \chi\otimes \varrho_{A,\ell}$
by the fact that $A$ has Frobenius traces concentrated in $K$.
\end{proof}

\section{Variants of Rajan's theorem}\label{section: Rajan}

In this section, relying on the results of \cite{Ser81}, we will prove variants of Corollary~ \ref{corollary: Rajan}, which will later play a crucial role in the proof of Theorem \ref{theorem: Main}. Resume the notations of \S\ref{section: Rajan first approach}.  
The following variant of Corollary \ref{corollary: Rajan} drops the condition that $\varrho$ and $\varrho'$ be strongly absolutely irreducible representations.  

\begin{theorem}\label{theorem: serre}
  Let $\varrho,\varrho'\colon G_k \rightarrow \GL_r(E_\lambda)$ be continuous, semisimple representations unramified outside a finite set $S\subseteq \Sigma_k$. Suppose that there exists a density~$1$ subset $\Sigma \subseteq \Sigma_k$ such that for every $\p \in \Sigma$ there exists $\epsilon_\p\in \{\pm 1\}$ such that
  \begin{equation}\label{equation: tracesuptosign}
  \Tr(\varrho'(\Frob_\p))= \epsilon_\p\cdot \Tr(\varrho(\Frob_\p))\,.
  \end{equation}
Suppose moreover that  
$$
\varrho\simeq \varrho_1\oplus \varrho_2\,,\qquad \varrho'\simeq \varrho_1'\oplus \varrho'_2\,,\qquad \varrho_i,\varrho_i':G_k\rightarrow \GL_{r_i}(E_\lambda)\,,
$$
where $\varrho_1$ is strongly absolutely irreducible, the Zariski closure of its image is connected, and satisfies
\begin{equation}\label{equation: no interaction}
\Hom_{G_F}(\varrho_1,\varrho_2)=\Hom_{G_F}(\varrho_1,\varrho'_2)=0
\end{equation} 
for every finite extension $F/k$. Then there exists a quadratic character $\chi$ of $G_k$ such that $\varrho'\simeq \chi\otimes \varrho$.
\end{theorem}  

\begin{proof}
By Theorem \ref{theorem: Rajan}, there exists a finite extension $L/k$ such that $\varrho|_L\simeq \varrho'_L$. By~\eqref{equation: no interaction}, we have $\varrho_i|_L\simeq \varrho_i'|_L$, and then the argument in the proof of Corollary \ref{corollary: Rajan} yields a character $\chi_1$ satisfying $\varrho_1'\simeq \chi_1\otimes \varrho_1$. 
For a $\p\in \Sigma$, we may rewrite \eqref{equation: tracesuptosign} as
$$
(\chi_1(\Frob_\p)-\epsilon_\p)\cdot \Tr(\varrho_1(\Frob_\p))-\epsilon_\p\cdot \Tr(\varrho_2(\Frob_\p))+\Tr(\varrho_2'(\Frob_\p))=0\,.
$$
Let $G_i$ (resp. $G_i'$) denote the Zariski closure of $\varrho_i(G_k)$ (resp. of $\varrho_i'(G_k)$) in $\GL_{r_i}$.  Denote by $X_1$ the finite set $\mathrm{im}(\chi_1)\pm 1 \subseteq \Qbar$. Choose a \emph{nonzero} $x_1$ in $X_1$ and $x_2$ in $\{\pm 1\}$. Let $W_{x_1,x_2}$ denote the subvariety of $G_1\times G_2\times G_2'$ cut by the equation
$$
x_1\cdot \Tr( g_1)+x_2\cdot \Tr(g_2)+\Tr(g_2')=0\qquad \text{for $g_1\in G_1$, $g_2\in G_2$, $g_2'\in G_2'$.}
$$
Note that $W_{x_1,x_2}\cap G_1$ is a closed subvariety of $G_1$. It is properly contained in~$G_1$, as otherwise we would obtain a relation among the traces of $\varrho_1$, $\varrho_2$, and $\varrho_2'$ that would contradict \eqref{equation: no interaction}. Hence, the connectedness of $G_1$ implies that $\dim(W_{x_1,x_2}\cap G_1)<\dim(G_1)$. Let~$H$ be the (finite) union of the closed proper subvarieties $W_{x_1,x_2}\cap G_1$ attached to choices of the $x_1,x_2$ as above.  Then, by \cite[Thm. 10, Thm. 8]{Ser81} the density of the set of primes $\p\in\Sigma$ for which $\varrho_1(\Frob_\p)$ belongs to $H$ is zero. We conclude that for every prime $\p$ in a subset of $\Sigma$, still of density $1$, we have $\chi_1(\Frob_\p)=\epsilon_\p$, which in particular implies that $\chi_1$ is quadratic. Therefore, for every~$\p$ in a density $1$ subset of $\Sigma_k$, we have $\Tr(\varrho'(\Frob_\p))= \chi_1(\Frob_\p)\cdot \Tr(\varrho(\Frob_\p))$, and the theorem follows from the semisimplicity of $\varrho$ and $\varrho'$.
\end{proof}

We will typically apply Theorem~\ref{theorem: serre} when~$\varrho$ and $\varrho'$ are $\varrho_{A,\ell}$ and $\varrho_{A',\ell}$, and $A$ has dimension $\leq 3$ and all endomorphisms defined over~$k$ (in this situation, the Zariski closure of $\varrho(G_k)$ is connected; see \cite[Prop. 2.17]{FKRS12}). 
The following is the most basic application.

\begin{corollary}\label{corollary: ellipticcurves}
Let $A$ and $A'$ be elliptic curves defined over $k$ for which \eqref{equation: main condition} holds. Then $A$ and $A'$ are quadratic twists.
\end{corollary}

\begin{proof}
As recalled in \S\ref{section: introduction}, the result is known if $\varrho_{A,\ell}$ is irreducible. If $A$ is an elliptic curve with CM defined over $k$, then $\varrho_{A,\ell}$ satisfies the hypotheses of Theorem \ref{theorem: serre}.
\end{proof}

The following result asserts that the local-global QT principle holds over the endomorphism field under a mild technical hypothesis. 

\begin{theorem}\label{theorem: mild refinement serre}
Suppose that \eqref{equation: main condition} holds for $A$ and $A'$, and that, if $K/k$ denotes the endomorphim field of $A$, then the Zariski closures of $\varrho_{A,\ell}(G_K)$ and $\varrho_{A',\ell}(G_K)$ are connected. Then there exists a quadratic character $\chi$ of $G_K$ such that $A'_K\sim A_{K,\chi}$ and such that the field extension $L/K$ cut out by $\chi$ satisfies that~$L/k$ is Galois.
\end{theorem}

\begin{proof}
Suppose that $A_K$ and $A'_K$ are not isogenous, as otherwise the result is clear.
Since the Zariski closures of $\varrho_{A,\ell}(G_K)$ and $\varrho_{A,\ell}(G_K)$ are connected, we can choose a prime $\ell$ such that
$$
\varrho_{A,\ell}|_K\simeq \bigoplus_i \varrho_i^{\oplus n_i}\,,\qquad \varrho_{A',\ell}|_K\simeq \bigoplus_i {\varrho'}_i^{\oplus n'_i}\,,
$$
where the $\varrho_i$ (resp. $\varrho_i'$) are strongly absolutely irreducible representations of $G_K$ pairwise nonisomorphic even after restriction to a finite extension of $K$. After reordering the $\varrho_i$, Corollary~\ref{corollary: locquad geomiso} provides a finite extension $L'/K$ such that $\varrho_i|_{L'}\simeq \varrho_i'|_{L'}$. In particular $n_i=n_i'$. 
There exists a density 1 subset $\Sigma\subseteq \Sigma_K$ such that for every $\p \in \Sigma$ there exists $\chi_\p\in \{\pm 1\}$ such that $\Tr(\varrho_{A',\ell}(\Frob_\p))=\chi_\p\cdot \Tr(\varrho_{A,\ell}(\Frob_\p))$.
The argument in the proof of Corollary \ref{corollary: Rajan} shows that there exists a character $\chi_1$ of $G_K$ such that $\varrho_1'\simeq \chi_1\otimes \varrho_1$. For every $\p \in \Sigma$, we thus have
$$
n_1(\chi_1(\Frob_\p)-\chi_\p)\cdot \Tr(\varrho_1(\Frob_\p))-\chi_\p\cdot \sum_{i\geq 2} n_i\Tr(\varrho_i(\Frob_\p))+ \sum_{i \geq 2} n_i\Tr(\varrho_i'(\Frob_\p))=0\,.
$$
Arguing as in the proof of Theorem~\ref{theorem: serre}, we see that $\varrho_{A',\ell}|_K\simeq \chi_1\otimes \varrho_{A,\ell}|_K$ and that $\chi_1$ is quadratic. If $L/K$ is the quadratic extension cut out by $\chi_1$, then there is an isogeny from $A_L$ to $A'_L$. Hence $L/k$ is the minimal extension over which all homomorphisms from $A_\Qbar$ to $A'_\Qbar$ are defined, and so $L/k$ is Galois.
\end{proof}

\section{Products of elliptic curves}\label{section: productsec}

In this section we consider abelian varieties which are $\Qbar$-isogenous to products of elliptic curves.  In~\S\ref{section: powersofec} we examine the case of $\Qbar$-isogenous factors. In this situation, Theorem~\ref{theorem: serre} fails to apply, but the local-global QT principle will follow by combining the Tate module tensor decompositions obtained in \cite{FG20} with a theorem of Ramakrishnan that we recall in \S\ref{section: Ramakrishnan}. Finally, in \S\ref{section: productsecCM} we consider the case of a product of pairwise geometrically nonisogenous elliptic curves with CM.

\subsection{Ramakrishnan's theorem}\label{section: Ramakrishnan}

Let $L/k$ be a finite Galois extension, $V$ a $\Qbar$-vector space, and $\theta\colon \Gal(L/k)\rightarrow \GL(V)$ an Artin representation. Let 
$$
\ad_\theta\colon \Gal(L/k)\rightarrow \Aut(\End(V))\,,\qquad \ad_\theta(s)(f):=\theta(s)\circ f\circ \theta(s)^{-1}
$$ 
denote the adjoint representation of $\theta$. It satifies $\ad_\theta \simeq \theta \otimes \theta^\vee$. Note that if $\ad^0_\theta$ denotes the restriction of $\ad_\theta$ on the subspace of trace~$0$ elements of $\End(V)$, then $\ad_\theta \simeq \triv \oplus \ad^0_\theta$, where $\triv$ denotes the trivial representation.

The next theorem of Ramakrishnan \cite[Thm. B]{Ram00} shows that, in the $2$-dimensional case, $\theta$ can be recovered from $\ad^0_\theta$ up to twist by a character (beware that Ramakrishnan's original theorem applies to general $\ell$-adic representations; we will only apply it to Artin representations).  

\begin{theorem}[Ramakrishnan]\label{theorem: Ramakrishnan}
If $\theta,\theta'\colon \Gal(L/k)\rightarrow \GL_2(\Qbar)$ are Artin representations with $\ad^0_\theta\simeq \ad^0_{\theta'}$, then there is a character $\chi$ of $\Gal(L/k)$ such that $\theta'\simeq \chi \otimes \theta$.  
\end{theorem}

\begin{proof}
It follows from Schur Lemma that $\theta$ is irreducible if and only if $\triv$ is not an irreducible constituent of $\ad^0_\theta$. Hence, $\theta'$ is irreducible if and only if $\theta$ is irreducible. See \cite[Thm. B]{Ram00} for the proof in the case that both $\theta$ and $\theta'$ are irreducible. If both $\theta$ and $\theta'$ are reducible, the proof is an elementary exercise left to the reader.
\end{proof}

\begin{remark}
If $\theta,\theta'\colon G_k\rightarrow \GL_3(\Qbar)$ are Artin representations such that $\ad^0_{\theta}\simeq \ad^0_{\theta'}$, it is not necessarily true that there exists a character $\chi$ such that $\theta'\simeq \chi \otimes \theta$. To construct an example of this, let $L/k$ be a degree 24 Galois extension with Galois group $\dih 4\times \cyc 3$, let $\nu$ denote a non-selfdual degree $2$ irreducible representation of $\Gal(L/k)$, and let $\omega$ denote the determinant of $\nu$. One easily verifies that $\theta=\triv \oplus \nu$ and $\theta'=\omega \oplus \nu$ provide the desired example.
\end{remark}

We will be concerned with the following consequence of the above theorem.

\begin{corollary}\label{corollary: polysuptosign}
Let $\theta,\theta'\colon \Gal(L/k)\rightarrow \GL_g(\Qbar)$ be Artin representations. Suppose that $g$ is $2$ or odd. If for every $s\in \Gal(L/k)$, there exists $\chi_s\in \{ \pm 1\}$ such that
\begin{equation}\label{equation: charpolycond}
\det(1-\theta'(s)T)=\det(1-\chi_s\cdot \theta(s)T)\,,
\end{equation}
then there exists a quadratic character $\chi\colon \Gal(L/k)\rightarrow \Q ^\times$ such that $\theta'\simeq \chi \otimes \theta$.  
\end{corollary}

\begin{proof}
By comparing the leading coefficients of both sides of~\eqref{equation: charpolycond}, one obtains 
\begin{equation}\label{equation: determinantrel}
\det(\theta'(s))=\chi_s^g\cdot \det(\theta(s))
\end{equation}
for every $s$ in $\Gal(L/k)$. If $g$ is odd, define the character $\chi:=\det(\theta')/\det(\theta)$. Since it satisfies $\chi(s)=\chi_s$ for every $s$, it is quadratic and $\theta'\simeq \chi \otimes \theta$ as desired.
Suppose next that $g=2$. Then \eqref{equation: charpolycond} implies that $\ad^0_{\theta}\simeq \ad^0_{\theta'}$. Hence by Theorem~\ref{theorem: Ramakrishnan} there exists a character $\chi\colon \Gal(L/k)\rightarrow \Qbar^\times$ such that $\theta'\simeq \chi\otimes \theta$. Taking determinants of this isomorphism and comparing with \eqref{equation: determinantrel}, one finds that $\chi$ is in fact quadratic.
\end{proof}

\begin{remark}\label{remark: dimension 6 counterexample}
The above corollary fails to hold for general even values of $g$. Let $L/k$ denote a biquadratic extension. Let $\varphi$ and $\psi$ denote two distinct nontrivial characters of $\Gal(L/k)$. It is easy to verify that the degree $6$ Artin representations 
$$
\theta:=\triv^{\oplus 3}\oplus \varphi \oplus \psi \oplus \varphi\psi\,,\qquad \theta':=\triv^{\oplus 2}\oplus \varphi^{\oplus 2}\oplus \psi^{\oplus 2}\,. 
$$
verify \eqref{equation: charpolycond}, while there is no character $\chi$ of $\Gal(L/k)$ such that $\theta' \simeq \chi \otimes \theta$. 
\end{remark}

\subsection{Tate module tensor decompositions}\label{section: powersofec}

In this section, we consider some families of geometrically isotypic abelian varieties, that is, abelian varieties that are $\Qbar$-isogenous to the power of an (absolutely) simple one. We will rely on the description of the Tate module of one such variety given in \cite[Thm. 1.1]{FG20}. We next restate that theorem in our particular cases of interest. 

Fix a complex conjugation $\sigma$ in $G_\Q$. Let $\varrho$ be an $\ell$-adic representation of~$G_k$ unramified outside a finite set $S\subseteq \Sigma_k$. We will denote by $\varrho^\sigma$ the $\ell$-adic representation defined by $\varrho^{\sigma}(s):=\varrho(\sigma s\sigma^{-1} )$ for every $s\in G_k$. Given a number field $F$, we will say that $\varrho$ is $F$-rational if $\det(1-\varrho(\Frob_\p)T)\in F[T]$ for every $\p\in \Sigma_k-S$. If $\theta$ is an Artin representation with coefficients in $F$, we will denote by~$\overline \theta$ the representation with coefficients in $\sigma(F)$ defined by $\thetabar(s):=\sigma(\theta(s))$. 

\begin{theorem}[\cite{FG20}]\label{theorem: FiteGuitart}
Suppose that $A_\Qbar$ is isogenous to the power of either: 
\begin{enumerate}[a)]
\item an elliptic curve without CM; or
\item an abelian surface with quaternionic multiplication (QM); or
\item an elliptic curve with CM by a quadratic imaginary field $M$.
\end{enumerate}
Then there exists a number field $F$ such that for every prime $\ell$ totally split in $F$:
\begin{enumerate}[i)]
\item If $a)$ or $b)$ hold, then 
$$
\varrho_{A,\ell}\simeq \theta \otimes_{\Q_\ell} \varrho\,,
$$
where $\theta\colon G_k \rightarrow \GL_g(F)$ is an Artin representation, and $\varrho\colon G_k\rightarrow \GL_2(\Q_\ell)$ is a strongly absolutely irreducible $F$-rational $\ell$-adic representation. 

\item If $c)$ holds, then $F$ contains $M$ and
$$
\varrho_{A,\ell}|_{kM}\simeq (\theta \otimes_{\Q_\ell} \chi)\oplus (\thetabar \otimes_{\Q_\ell} \chibar)\,,
$$ 
where $\theta\colon G_{kM} \rightarrow \GL_g(F)$ is an Artin representation and $\chi,\chibar\colon G_{kM}\rightarrow \Q_\ell^\times$ are $F$-rational continuous characters. Moreover, the $\ell$-adic representations $\theta\otimes_{\Q_\ell}\chi$ and $\thetabar\otimes_{\Q_\ell}\chibar$ are in fact $M$-rational. Finally, if $k\not = kM$, then 
\begin{equation}\label{equation: sigmatrace}
\varrho_{A,\ell}\simeq \Ind^k_{kM}(\theta \otimes_{\Q_\ell} \chi)\qquad \text{and}\qquad \thetabar\otimes_{\Q_\ell}\chibar|_{kM}\simeq \theta^\sigma\otimes_{\Q_\ell}\chi^\sigma|_{kM}\,.
\end{equation}
\end{enumerate} 
\end{theorem}

\begin{remark}\label{remark: decompforoddg}
There is a particular situation of case $i)$ of the above theorem in which the representations $\theta,\varrho$ admit especially explicit descriptions. Suppose that there exists an elliptic curve $B$ without CM and \emph{defined over $k$} such that $A_\Qbar$ and $B^g_\Qbar$ are isogenous (by \cite[Thm. 2.21]{FG18}, this happens for example whenever $g$ is odd). Then we may take $F=\Q$, $\varrho=\varrho_{B,\ell}$, and $ \theta=\Hom^0(B^g_\Qbar,A_\Qbar)$,  
in which case the theorem follows from \cite[Thm. 3.1]{Fit12}. 
\end{remark}

The above theorem is a crucial input for Theorems \ref{theorem: nonCMpower} and \ref{theorem: CMpower}, whose proofs will also need the following auxiliary lemma on the abundance of ordinary primes. Let~$\p$ be a prime of $\Sigma_k$ of good reduction for $A$ lying over the rational prime~$p$. We say that $\p$ is \emph{ordinary} if the central coefficient of $L_\p(A,T)$ is not divisible by $p$, equivalently, if $g$ of the roots of $L_\p(A,T)$ have $p$-adic valuation~$0$; we say that $\p$ is \emph{supersingular} if all the roots of $L_\p(A,T)$ are of the form $\zeta \Nm(\p)^{-1/2}$, where $\zeta$ is a root of unity and $\Nm(\p)$ is the absolute norm of $\p$. 

\begin{lemma}\label{lemma: ordinary}
Suppose that~$A_\Qbar$ is isogenous to the power of an elliptic curve $B$. If~$B$ has CM, say by a quadratic imaginary field $M$, suppose that $k=kM$. Then the set of primes of ordinary reduction for $A$ has density $1$. 
\end{lemma}

\begin{proof}
We claim that a prime $\p$ of $\Sigma_k$ of good reduction for~$A$ is either ordinary or supersingular. Indeed, let $q$ denote $\Nm(\p)$, let $\alpha_\p$ be a root of $L_\p(A,T)$, and let $v\colon \Qbar_p^\times\rightarrow \Q$ denote the $p$-adic valuation normalized so that $v(q)$ is $1$. By hypothesis, there exists a finite Galois extension $K/k$ such that $A_K$ is isogenous to the power of an elliptic curve over $K$. If $f\geq 1$ denotes the residue degree of $\p$ in $K/k$, then $A_\p\times_{\F_q}\F_{q^f}$ is isogenous to the power of an elliptic curve over $\F_{q^f}$, and hence the valuation $v(\alpha_\p^f)$ is either $0$,$f$, or $f/2$. In the first two cases, $\p$ is ordinary, and in the latter case $\p$ is supersingular. We need to show that the set $S$ of supersingular primes has density $0$.
If $B$ does not have CM, let $\theta$ and $\varrho$ be as in part $i)$ of Theorem~\ref{theorem: FiteGuitart}. If $\p\in S$, then $\Tr(\varrho(\Frob(\p)))$ is divisible by $\sqrt q$ and hence is limited to finitely many possibilities. By \cite{Ser81}, this implies that $S$ has density $0$. If $B$ has CM, let~$\chi$ be as in part $ii)$ of Theorem \ref{theorem: FiteGuitart}. If $\p\in S$, then $\chi(\Frob_\p)=\zeta \sqrt{q}$, where $\zeta$ is a root of unity whose order is bounded in terms of $g$. By \cite{Ser81} or \cite{Hec20}, and under the assumption $k=kM$, this implies that~$S$ has density $0$.   
\end{proof}

\begin{remark}
When $g$ is odd and $B$ has no CM, the above lemma admits an even simpler proof. In this case, by Remark \ref{remark: decompforoddg}, there is an elliptic curve $B$ \emph{defined over~$k$} and an Artin representation $\theta$ such that $\varrho_{A,\ell}\simeq \theta\otimes \varrho_{B,\ell}$. Thus a prime of good reduction of $A$ is ordinary for $A$ if and only if it is ordinary for $B$. That the set of primes ordinary for $B$ has density $1$ is well known.  
\end{remark}

\begin{theorem}\label{theorem: nonCMpower}
Suppose that \eqref{equation: main condition} holds for $A$ and $A'$, and that:
\begin{enumerate}[i)]
\item The dimension $g$ of $A$ is either odd or equal to $2$.
\item $A_\Qbar$ is either an abelian surface with QM or isogenous to the power of an elliptic curve without CM.
\end{enumerate}
Then $A$ and $A'$ are quadratic twists.
\end{theorem}

\begin{proof}
By Corollary \ref{corollary: locquad geomiso} there is a finite Galois extension $L/k$ such that $A_L$ and~$A'_L$ are isogenous. Hence $A'$ also satisfies hypothesis $ii)$.
After possibly enlarging $L/k$, by Theorem \ref{theorem: FiteGuitart}, there are a number field $F$, Artin representations $\theta,\theta'$ of $\Gal(L/k)$ realizable over $F$, and $F$-rational $\ell$-adic representations~$\varrho,\varrho'$ of $G_k$ of degree $2$ such that $\varrho_{A,\ell}\simeq \theta \otimes_{\Q_\ell} \varrho$ and $\varrho_{A',\ell}\simeq \theta' \otimes_{\Q_\ell} \varrho'$,
where $\ell$ is a prime totally split in $F$. Since $\varrho|_L\simeq \varrho'|_L$ and $\varrho$ is strongly absolutely irreducible, there is a character $\chi$ of $\Gal(L/k)$ such that $\varrho'\simeq \chi \otimes \varrho$.
If $\p\in \Sigma_k$ is a prime of good reduction for $A$, let $\alpha_\p$ and $\beta_\p$ be the eigenvalues of $\varrho(\Frob_\p)$.
By \cite[Thm. 3]{Saw16} and Lemma \ref{lemma: ordinary}, there exists a density $1$ subset $\Sigma \subseteq \Sigma_k$ of primes of good reduction for $A$ and $A'$ such that for every $\p\in \Sigma$ the quotient $\alpha_\p/\beta_\p$ is not a root of unity, and 
$$
\det(1-(\theta' \otimes \chi \otimes \varrho)(\Frob_\p) T)=\det(1-\epsilon_\p\cdot (\theta\otimes \varrho)(\Frob_\p) T)
$$ 
for some $\epsilon_\p\in \{\pm 1\}$. Hence for every $\p\in \Sigma$ we have
$$
\det(1-\alpha_\p\cdot (\theta' \otimes \chi)(\Frob_\p) T)=\det(1-\epsilon_\p\alpha_\p\cdot \theta(\Frob_\p) T)\,.
$$
Since $\Sigma$ has density $1$, for every $s\in \Gal(L/k)$ there must be $\epsilon_s\in \{\pm 1\}$ such that 
$$
\det(1- (\theta' \otimes \chi)(s) T)=\det(1-\epsilon_s\cdot \theta(s) T)\,.
$$
Since $i)$ holds, by Corollary \ref{corollary: polysuptosign}, there exists a quadratic character $\varphi$ of $\Gal(L/k)$ such that $\theta'\otimes \chi\simeq \varphi \otimes \theta$. Therefore $\varrho_{A',\ell}\simeq \theta'\otimes \chi\otimes \varrho\simeq \varphi\otimes \theta\otimes \varrho\simeq \varphi \otimes \varrho_{A,\ell}$.
\end{proof}

\begin{remark}\label{remark: dimension 6 counterexample AV}
We note that the theorem above is not true for general even values of $g$. Let $E$ be an elliptic curve defined over $k$ without CM. Let $A$ and $A'$ be the 6-dimensional abelian varieties $E\otimes \theta$ and $E\otimes \theta'$, where $\theta$ and $\theta'$ are the Artin representations defined in Remark \ref{remark: dimension 6 counterexample}. Then $A$ and $A'$ are locally quadratic twists at almost all primes of $\Sigma_k$. However, if there were a quadratic character $\varepsilon$ of $G_k$ such that $\varrho_{A',\ell} \simeq \varepsilon \otimes \varrho_{A,\ell}$, then, arguing as in the proof of the above theorem, we would obtain that $\theta'$ is isomorphic to $\varepsilon \otimes \theta$, which contradicts Remark \ref{remark: dimension 6 counterexample}.
\end{remark}

\begin{theorem}\label{theorem: CMpower}
Suppose that \eqref{equation: main condition} holds for $A$ and $A'$, and that:
\begin{enumerate}[i)]
\item The dimension $g$ of $A$ is either odd or equal to $2$.
\item $A_\Qbar$ is isogenous to the power of an elliptic curve with CM.
\end{enumerate}
Then $A$ and $A'$ are quadratic twists.
\end{theorem}

\begin{proof}
Let $M$ denote the quadratic imaginary field associated with the elliptic factor of $A$. Suppose first that $M$ is contained in $k$.
By arguing as in the proof of Theorem~\ref{theorem: nonCMpower}, using Theorem~\ref{theorem: FiteGuitart} we can find a finite Galois extension $L/k$, a number field $F$, Artin representations $\theta, \theta'\colon \Gal(L/k)\rightarrow \GL_g(F)$, and continuous $F$-rational $\ell$-adic characters $\chi,\chi',\overline\chi,\overline\chi'$ of $G_k$ such that 
\begin{equation}\label{equation: CMdecomp}
\varrho_{A,\ell}\simeq (\theta \otimes \chi) \oplus (\thetabar \oplus \chibar)\,,\quad\text{and}\qquad \varrho_{A',\ell}\simeq (\theta' \otimes \chi') \oplus ( \thetabar' \oplus \chibar')\,,
\end{equation}
where $\ell$ is a prime totally split in $F$. After reordering $\chi$ and $\chibar$ if necessary, we may assume that $\chi|_L\simeq \chi'|_L$. Therefore there exists a character $\psi$ of $\Gal(L/k)$ such that $\chi'\simeq \psi \chi$. 
By Lemma~\ref{lemma: ordinary} there exists a density $1$ subset $\Sigma\subseteq \Sigma_k$ of primes of good reduction for~$A$ and~$A'$ such that for every $\p \in \Sigma$ the quotient $\chi(\Frob_\p)/\chibar(\Frob_\p)$ is not a root of unity, and 
$$
\det(1-(\theta' \otimes \psi\chi \oplus \thetabar' \otimes \psibar  \chibar)(\Frob_\p) T)=\det(1-\epsilon_\p\cdot (\theta \otimes \chi \oplus \thetabar \otimes \chibar)(\Frob_\p) T)
$$ 
for some $\epsilon_\p\in \{\pm 1\}$. Hence for every $\p\in \Sigma$ we have
$$
\det(1-(\theta' \otimes \psi \chi )(\Frob_\p) T)=\det(1-\epsilon_\p\cdot (\theta \otimes \chi )(\Frob_\p) T)\,.
$$
Since $\Sigma$ has density $1$, for every $s\in \Gal(L/k)$ there must be $\epsilon_s\in \{\pm 1\}$ such that
$$
\det(1-(\theta' \otimes \psi )(s) T)=\det(1-\epsilon_s\cdot \theta (s) T).
$$
Corollary \ref{corollary: polysuptosign} yields a quadratic character $\varphi$ of $G_k$ such that $\theta'\otimes \psi\simeq \varphi \otimes \theta$.~Thus 
$$
\varrho_{A',\ell}\simeq (\theta'\otimes \psi\chi)\oplus (\thetabar'\otimes \psibar\chibar) \simeq \varphi\otimes (\theta\otimes \chi \oplus  \thetabar\otimes \chibar)\simeq \varphi \otimes \varrho_{A,\ell}\,.
$$
Suppose now that $M$ is not contained in $k$.
Arguing as in the previous case, we see that there are a continuous character $\chi \colon G_{kM}\rightarrow \Q_\ell^\times$, a quadratic character $\varphi$ of $\Gal(L/kM)$, and an Artin representation $\theta$ of $\Gal(L/kM)$ such that  
$$
\varrho_{A,\ell}\simeq \Ind_{kM}^k(\theta\otimes \chi)\,,\quad \varrho_{A',\ell}\simeq \Ind_{kM}^k(\varphi\otimes \theta\otimes \chi)\,,\quad \varphi^\sigma\otimes\theta^\sigma \simeq \varphi\otimes\theta^\sigma \,.   
$$
The third of the above isomorphisms follows from
$$
\varphi^\sigma\otimes\theta^\sigma \otimes \chi^\sigma \simeq
\overline \varphi\otimes \overline \theta \otimes \overline \chi \simeq
\varphi\otimes \overline \theta \otimes \overline \chi \simeq
 \varphi\otimes\theta^\sigma \otimes \chi^\sigma\,,
$$
where we have used \eqref{equation: sigmatrace}. Suppose that $g$ is odd. Taking determinants of $\varphi^\sigma\otimes\theta^\sigma \simeq \varphi\otimes\theta^\sigma$, we get $(\varphi^{\sigma})^g=\varphi^g$, that is, $\varphi^{\sigma}$ and $\varphi$ coincide as characters of $\Gal(L/kM)$. Hence $\varphi$ extends to a character $\tilde\varphi$ of $\Gal(L/k)$, and  
$$
\varrho_{A',\ell}\simeq \tilde\varphi\otimes\Ind_{kM}^k( \theta\otimes \chi)\simeq \tilde \varphi \otimes \varrho_{A,\ell}\,.
$$
Taking determinants of the above isomorphism, we see that $\tilde\varphi^{2g}=1$. Since trivially we also have $\tilde\varphi^ 4=1$, the fact that $g$ is odd implies that $\tilde\varphi$ is quadratic.
If $g=2$, then apply Lemma \ref{lemma: greatly simplifying lemma} below taking $N=kM$.  
\end{proof} 

The next lemma is applied in the proof of the above theorem in the case $g=2$. We state and prove it also in the case $g=3$ for future reference. 

\begin{lemma}\label{lemma: greatly simplifying lemma}
Suppose that \eqref{equation: main condition} holds for $A$ and $A'$, and that:
\begin{enumerate}[i)]
\item $A$ and $A'$ have dimension $\leq 3$.
\item There exist a quadratic extension $N/k$ contained in the endomorphism field $K/k$ and a quadratic character $\chi\colon G_N\rightarrow \{\pm 1\}$ such that $A_{N,\chi}\sim A'_N$.
\item $A$ has Frobenius traces concentrated in $N$.
\end{enumerate} 
Then $A$ and $A'$ are quadratic twists.
\end{lemma}

\begin{proof}
Let $E/N$ denote the field extension cut out by $\chi$. Let $E^{\Gal}/k$ denote the Galois closure of $E/k$. Since $E/k$ contains the quadratic subextension $N/k$, the possibilities for $\Gal(E^{\Gal}/k)$ are $\cyc 2\times \cyc 2$, $\cyc 4$, or $\dih 4$. In the first two cases, the character~$\chi$ extends to a character $\tilde \chi$ of $\Gal(E^{\Gal}/k)$ such that $\tilde\chi^ 4=1$. By $iii)$, we have that $\varrho_{A',\ell}\simeq \tilde\chi\otimes \varrho_{A,\ell}$. If $g=3$, then $\tilde \chi^6=1$, and hence $\tilde \chi$ must be quadratic. 
If $g=2$, the lemma follows from Proposition~\ref{proposition: iftwistquadratic} and Lemma~\ref{lemma: Cnextensions} (note that $\tilde \chi|_K$ is indeed quadratic). Suppose finally that $\Gal(E^{\Gal}/k)\simeq \dih 4$. Then there exists a biquadratic extension $F/k$ such that $FN=E^{\Gal}$. 
Therefore, there exists a (quadratic) character $\psi$ of $\Gal(F/k)$ such that $\psi|_N=\chi|_N$. By $iii)$, we have that $\varrho_{A',\ell}\simeq \psi\otimes \varrho_{A,\ell}$, which completes the proof of the lemma.
\end{proof}

\subsection{Products of pairwise geometrically nonisogenous elliptic curves}\label{section: productsecCM}
In this section we prove that the local-global QT principle holds for products of pairwise geometrically nonisogenous elliptic curves. We resume the notations from the previous section.
The following lemma will be used in this and later sections.

\begin{lemma}\label{lemma: splittingoverk}
Let $A$ be an abelian variety defined over $k$. Suppose that $A_\Qbar$ is isogenous to the product $B \times C$ of abelian varieties $B,C$ defined over $\Qbar$ such that
\begin{equation}\label{equation: nointertwining}
\Hom(B,C^s)=0
\end{equation}
for all $s \in G_{k}$. Then there exist abelian varieties $B',C'$ defined over $k$ such that $A$ is isogenous to $B' \times C'$.
\end{lemma}

\begin{proof}
Choose an isogeny $\varphi \colon B\times C\rightarrow A_\Qbar$. Let $B'$ denote the abelian subvariety of $A_\Qbar$ generated by $\varphi(B)$ and its Galois conjugates $\varphi(B)^s$ for $s \in G_k$. Define $C'$ analogously. Note that $B'$ and $C'$ are defined over $k$. By \eqref{equation: nointertwining}, we have $\Hom(B',C')=0$, and hence the addition map $\psi \colon B'\times C'\rightarrow A$ has finite kernel. Therefore
$$
\dim(A)\geq \dim(B')+\dim(C')\geq \dim(B)+\dim(C)=\dim(A)\,.
$$
Hence $A$ and $B'\times C'$ have the same dimension, and thus $\psi$ is an isogeny.
\end{proof}

By the above lemma and Theorem~\ref{theorem: serre}, the local-global QT principle holds for products of pairwise geometrically nonisogenous elliptic curves one of whose factors does not have CM. From now on, we focus on the case that all factors have CM.

\begin{lemma}\label{lemma: LpolyprodCM}
Suppose that $A$ is isogenous to the product of $g$ elliptic curves $E_i$ with CM, say by $M_i$, and pairwise not $\Qbar$-isogenous. Then there is a subset $\Sigma\subseteq \Sigma_k$ of density $1$ consisting of primes of absolute residue degree $1$ and of good reduction for $A$ such that for every $\p\in \Sigma$, of norm $p$, we have 
\begin{equation}\label{equation: localfactorprodCM}
L_\p(A,T)=\prod_{i=1}^g(1-\alpha_i T)(1-p\alpha_i^{-1} T)\,,
\end{equation} 
where the reciprocal roots $\alpha_i\in \Qbar$ satisfy:
\begin{enumerate}[i)]
\item If $\p$ splits in $M_i$, then $\alpha_i \in M_i-\Q$ and $\alpha_i/\sqrt p$ is not a root of unity.\vspace{0,1cm}
\item If $\p$ is inert in $M_i$, then $\alpha_i=\sqrt{-p}$ and $\alpha_i\not \in M_j$ for any $j$.
\end{enumerate}
\end{lemma}

\begin{proof}
If $\p$ splits in $kM_i$, consider the decomposition $\varrho_{E_i,\ell}|_{kM}\simeq \chi \oplus \overline \chi$ 
from Theorem~\ref{theorem: FiteGuitart} noting that $\chi$ is an $M_i$-rational Hecke character in this case; if~$\p$ is inert in $kM_i$, then use the description of $\varrho_{E_i}$ as the induction of $\chi$ from $kM_i$ down to $k$. We spare the details to the reader. 
\end{proof}

\begin{proposition}\label{proposition: productsec}
Suppose that \eqref{equation: main condition} holds for $A$ and $A'$, and that $A_\Qbar$ is isogenous to the product of $g$ nonisogenous elliptic curves with CM. Then $A$ and $A'$ are quadratic twists.
\end{proposition}

\begin{proof}
By Corollary \ref{corollary: locquad geomiso} and Lemma \ref{lemma: splittingoverk}, there exist elliptic curves $E_i,E'_i$ defined over $k$ such that $E_{i,\Qbar}$ and $E'_{i,\Qbar}$ are isogenous, and $A\sim \prod_i E_i$ and $ A'\sim \prod_i E_i'$.
Suppose first that $g=2$. Lemma \ref{lemma: splittingoverk} also implies that if $M_i$ denotes the CM field of $E_i$, then the endomorphism field $K$ of $A$ is $kM_1 M_2$.  From the description of $L_\p(A,T)$ from Lemma~\ref{lemma: LpolyprodCM}, we see that $E_i$ and $E_i'$ satisfy \eqref{equation: main condition}. Then, by Corollary~\ref{corollary: ellipticcurves}, there are quadratic characters $\varphi_i$ of $G_k$ such that $\varrho_{E_i',\ell}\simeq \varphi_i \otimes \varrho_{E_i,\ell}$. By Theorem \ref{theorem: mild refinement serre}, we must have $\varphi_1|_K=\varphi_2|_K$. Let now $\chi_i$ denote the nontrivial character of $\Gal(kM_i/k)$ if this group is nontrivial, and the trivial character otherwise. There exist integers $\delta_1,\delta_2\in \{0,1\}$ such that $\varphi_2=\chi_1^{\delta_1}\chi_2^{\delta_2}\varphi_1$. Then
$$
\varrho_{A',\ell}\simeq (\varphi_1\otimes \varrho_{E_1,\ell}) \oplus (\chi_1^{\delta_1}\chi_2^{\delta_2}\varphi_1\otimes\varrho_{E_2,\ell})\simeq \chi_1^{\delta_1}\varphi_1\otimes (\varrho_{E_1,\ell} \oplus \varrho_{E_2,\ell})\simeq \chi_1^{\delta_1}\varphi_1\otimes \varrho_{A,\ell}\,,
$$
where we have used that $\chi_i\otimes \varrho_{E_i,\ell}\simeq \varrho_{E_i,\ell}$.

Suppose now that $g\geq 3$. By hypothesis, there exists a density 1 subset $\Sigma \subseteq \Sigma_k$ such that for every $\p \in \Sigma$ there exists $\epsilon_\p\in \{\pm 1\}$ such that $L_\p(A',T)=L_\p(A,\epsilon_\p T)$.
Let $A_i:=E_1 \times E_i$ and $A_i':=E_1'\times E_i'$ for every $i\geq 2$. From the description of $L_\p(A,T)$ from Lemma~\ref{lemma: LpolyprodCM}, by shrinking $\Sigma$ if necessary, we may assume that 
$L_\p(A'_i,T)=L_\p(A_i,\epsilon_\p T)$
for every $\p\in \Sigma$.
By the $g=2$ case, there exists a quadratic character $\varphi_i$ of $G_k$ such that $A_i'\sim A_{i,\varphi_i}$. Hence 
\begin{equation}\label{equation: igualtat anodina}
\epsilon_\p\cdot \Tr(\varrho_{A_i,\ell}(\Frob_\p))=\Tr(\varrho_{A'_i,\ell}(\Frob_\p))=\varphi_i(\Frob_\p)\cdot\Tr(\varrho_{A_i,\ell}(\Frob_\p))
\end{equation}
for every $\p \in \Sigma$. Moreover, by Hecke's equidistribution theorem, there is a density~1 subset $\Sigma'\subseteq \Sigma$ such that for every $\p \in \Sigma'$, we have $\Tr(\varrho_{A_i,\ell}(\Frob_\p))\not=0$ unless $\p$ is both inert in~$M_1$ and~$M_i$. Let $\Sigma_i\subseteq \Sigma'$ be the subset of those $\p$ which are split in $M_1$ or split in~$M_i$. Note that $\Sigma_2 \cap \Sigma_i$ has density equal to $5/8$. 
From \eqref{equation: igualtat anodina}, for every $\p  \in \Sigma_2 \cap \Sigma_i$, we have
$\varphi_i(\Frob_\p)=\epsilon_\p=\varphi_2(\Frob_\p)$.
Since $\varphi_i$ and $\varphi_2$ are quadratic and coincide on a set of Frobenius elements of density~$>1/2$, they must coincide. Hence $A'\sim \prod_i E_i' \sim \prod_{i} E_{i,\varphi_2}\sim A_{\varphi_2}$.
\end{proof}

\section{Proof of the main theorem}\label{section: proof}

The absolute type of an abelian variety $A$ is the isomorphism class of the $\R$-algebra $\End(A_\Qbar)\otimes_\Z\R$. We borrow from \cite[\S4.1]{FKRS12} the labels $\Ab,\dots, \Fb$ for the absolute type of an abelian surface. See \cite[\S3.2.1]{FKS21a} or \cite[\S3.5]{FKS21c} for the description of the absolute types $\Ab,\dots, \Nb$ of an abelian threefold. Note that if \eqref{equation: main condition} holds for $A$ and~$A'$, then they have the same absolute type by Corollary~\ref{corollary: locquad geomiso}. Recall that they also have the same endomorphism field by Lemma \ref{lemma: endomorphismfield}.
To complete the proof of Theorem \ref{theorem: Main} we will distinguish the cases $g=2$ and $g=3$.
Before, we recall a construction that will be used in several proofs (in fact, it has already appeared implicitly via Theorem~\ref{theorem: FiteGuitart}). We refer to \cite[Chap. II]{Rib76} for further details.

\begin{remark}\label{remark: lambda adic reps}
Let $B$ be an abelian variety defined over $k$ of dimension $g$, and let $M$ be a number field. Suppose that there is a $\Q$-algebra embedding $M\hookrightarrow \End^0(B)$. 
Choose a prime $\ell$ totally split in $M$, so that the $[M:\Q]$ embeddings $\lambda_i\colon M \hookrightarrow \Qbar_\ell$, take values in $\Q_\ell$.  
Define 
\begin{equation}
V_{\lambda_i}(B):=V_\ell(B)\otimes_{M\otimes \Q_\ell,\lambda_i}\Q_\ell\,,
\end{equation}
where $\Q_\ell$ is being regarded as an $M\otimes \Q_\ell$-module via~$\lambda_i$. It has dimension $2g/[M:\Q]$ as a vector space over $\Q_\ell$. We let $G_k$ act naturally on $V_\ell(B)$ and trivially on $\Q_\ell$. With this action, there is an isomorphism $V_\ell(B)\simeq \bigoplus_i V_{\lambda_i}(B)$.
\end{remark}

\subsection{Abelian surfaces} 
Throughout this section assume $g=2$. Theorem \ref{theorem: Main} follows from Theorem \ref{theorem: serre} or Corollary~\ref{corollary: trivend} if the absolute type of~$A$ and~$A'$ is \Ab, from Theorem \ref{theorem: serre} if it is \Cb, from Theorem~\ref{theorem: nonCMpower} if it is \Eb, and from Theorem~\ref{theorem: CMpower} if it is \Fb. We will complete the proof of Theorem~\ref{theorem: Main} in the remaining cases. 

\subsubsection{Absolute type \Db}\label{section: TypeD1} 
The case that $A_\Qbar$ is the product of two elliptic curves is covered by Proposition \ref{proposition: productsec}. Assume from now on that the geometric endomorphism algebra of $A$ is a quartic CM field, which we will call~$M$. By \cite[p. 515, Proposition 3]{Shi71} and \cite[p.64]{Shi98}, one of the following three cases occurs:
\begin{enumerate}[i)]
\item $K=k$ and $\End^0(A) \simeq M$.
\item $K/k$ is quadratic and $\End^0(A)$ is a real quadratic field.
\item $K/k$ is cyclic of degree $4$ and $\End^0(A)\simeq \Q$. 
\end{enumerate}
Let $\ell$ be a prime totally split in $M$, and let $\lambda_1,\dots,\lambda_4$ be the embeddings of $M$ into~$\Q_\ell$. The following lemma is an easy consequence of Faltings isogeny theorem.

\begin{lemma}\label{lemma: Vi stronglyabsirred}
We have $\Hom_{G_K}(V_{\lambda_i},V_{\lambda_j})=0$ for $i\not=j$, and $\End_{G_K}(V_{\lambda_i})\simeq \Q_\ell$.
\end{lemma}

Set $n=4/[K:k]=[\End^ 0(A):\Q]$. By \cite[Thm. 3, Rem. 2, p. 186]{Mil72} (see also the first part of the proof of \cite[Thm. 2.9]{FG20}), there is an isomorphism 
\begin{equation}\label{equation: moduledescription}
V_\ell(A)\simeq \bigoplus_{i=1}^n \Ind_K^k(V_{\lambda_i})
\end{equation}
 of $G_k$-modules, where we assume that $\lambda_1,\dots,\lambda_n$ have been ordered so that their restrictions to $\End^0(A)$ are all distinct.
Note that~$A$ falls in case i) (resp. ii), iii)) if and only if so does~$A'$. 

\begin{proposition}\label{proposition: complexmultiplication}
If $\End^0(A_\Qbar)$ is a quartic CM field, then Theorem \ref{theorem: Main} holds.
\end{proposition}

\begin{proof}
In case i), the proposition follows from Theorem \ref{theorem: serre}. Note that $\varrho_{A,\ell}$ satisfies the hypotheses of this theorem by Lemma \ref{lemma: Vi stronglyabsirred}. In case ii) (resp. iii)), the proposition follows from case i) (resp. ii)) and Lemma \ref{lemma: greatly simplifying lemma}.
Note that we can apply this lemma since $A$ has Frobenius traces concentrated in $K$ as it follows from \eqref{equation: moduledescription}. 
\end{proof}

\begin{remark}
An alternative way to prove Proposition \ref{proposition: complexmultiplication} in case i) consists in choosing a prime $\ell$ inert in $M$. Then $V_\ell(A)$ is an $M_\lambda$-module of dimension $1$, and hence strongly absolutely irreducible. Therefore there exists a character $\varphi$ of $G_k$ such that $\varrho_{A',\ell}\simeq \varphi \otimes \varrho_{A,\ell}$, and then the proposition follows from Proposition \ref{proposition: iftwistquadratic}.
\end{remark}

\subsubsection{Absolute type \Bb}\label{section: TypesBC}

Suppose that $A$ has absolute type \Bb, that is, $A_\Qbar$ is either:
\begin{enumerate}[i)]
\item isogenous to the product of two nonisogenous non CM elliptic curves, or 
\item simple and $M:=\End^0(A_\Qbar)$ is a real quadratic field.
\end{enumerate}
The endomorphism field $K/k$ is at most quadratic. If $K=k$, then Theorem \ref{theorem: Main} follows from Theorem \ref{theorem: serre}. Assume that $K/k$ is quadratic. Let $\ell$ be a prime, which we assume split in~$M$ if we are in case i). Then, there are strongly absolutely irreducible $\ell$-adic representations $\varrho,\varrho'\colon G_K\rightarrow \GL_2(\Q_\ell)$ such that
$\varrho_{A,\ell}\simeq \Ind_K^k (\varrho)$ and $\varrho_{A',\ell}\simeq \Ind_K^k (\varrho')$.
In case i), we have that $A_K$ is isogenous to the product of two elliptic curves $E_1, E_2$ defined over $K$, and we may take $\varrho=\varrho_{E_i,\ell}$; in case ii), let $\varrho$ be the representation afforded by $V_\lambda(A)$, where $\lambda$ is an embedding of~$M$ into~$\Q_\ell$. Theorem \ref{theorem: Main} then follows from the case $K=k$ and Lemma \ref{lemma: greatly simplifying lemma}. 

\subsection{Abelian threefolds}

In this section we assume that $g=3$. If the absolute type of $A$ is \Ab, then Theorem \ref{theorem: Main} follows from Corollary~\ref{corollary: trivend} or Theorem~\ref{theorem: serre}; if it is \Mb, it follows from Theorem~\ref{theorem: nonCMpower}; if it is \Nb, it follows from Theorem~\ref{theorem: CMpower}.

\begin{lemma}
If $A$ has absolute type \Cb, \Db, \Gb, \Ib, or \Kb, then Theorem \ref{theorem: Main} holds.
\end{lemma}

\begin{proof}
By Lemma \ref{lemma: splittingoverk}, there exist abelian varieties $B,B'$ defined over $k$ such that~$A$ is isogenous to $B\times B'$ and $\End(B_\Qbar) \simeq \Z$. By Faltings isogeny theorem, there exist $\Q_\ell[G_k]$-modules $V,V'$ such that $V_\ell(A)\simeq V \oplus V'$ and $\End_{G_F}(V)\simeq \Q_\ell$ for every finite extension $F/k$. Then the lemma follows from Theorem~\ref{theorem: serre}.
\end{proof}

The next two results consider abelian threefolds that contain an elliptic factor~$E$ with CM by $M$ that is not geometrically isogenous to any other factor.

\begin{proposition}\label{proposition: keycaseL}
Let $B$ be an abelian surface defined over $k$, and $\chi$ and $\psi$ be quadratic characters. Suppose that $\Hom(E_\Qbar,B_{\Qbar})=0$ and that \eqref{equation: main condition} holds for $A:=E\times B$ and $A':=E_\chi\times B_\psi$. Then $A$ and $A'$ are quadratic twists.
\end{proposition}

\begin{proof}
We may reduce to the case that $\chi=1$. Indeed, this particular case of the proposition applies to the pair $A$ and $A'_\chi$, and $A$ and $A'_\chi$ are quadratic twists if and only if so are $A$ and $A'$. Assume henceforth that $\chi=1$.

We will assume that $k\not=kM$, as otherwise the proposition follows from Theorem~\ref{theorem: serre}. Also by Theorem~\ref{theorem: serre}, there is a quadratic character $\varepsilon$ of $G_{kM}$ such that $A_{kM,\varepsilon}\sim A'_{kM}$. In fact, the proof of the theorem shows that $\varepsilon$ satisfies $E_{kM}\sim E_{kM,\varepsilon}$, which implies that $\varepsilon$ must be trivial. Let $\varphi$ denote the nontrivial character of $\Gal(kM/k)$. We will assume that $\psi\not=1$, as otherwise there is nothing to show. We will also assume that $\varphi\not=\psi$, as otherwise the result follows from $A'\sim E\times B_\varphi \sim E_\varphi \times B_\varphi \sim A_\varphi$.
By taking the restriction of scalars of $A_{kM}\sim A'_{kM}$ from $kM$ to $k$, we obtain
\begin{equation}\label{equation: restriction of scalars}
B\times B_\varphi \sim B_\psi \times B_{\psi\varphi}\,.
\end{equation}
We will distinguish three cases:
a) $B$ is simple (over $k$);
b) $B\sim C\times D$, where $C$ and $D$ are elliptic curves over $k$ which are not quadratic twists.
c) $B\sim C\times C_\xi$, where $C$ is an elliptic curve over $k$ and $\xi$ is a quadratic character. 

Suppose that a) holds. From \eqref{equation: restriction of scalars} we see that either $B\sim B_\psi$ or $B_\varphi\sim B_{\psi}$. In the first case, we obtain that $A\sim A'$, and in the latter case we have
$$
A' \sim E\times B_\psi \sim E_\varphi\times B_{\varphi} \sim   A_\varphi\,.
$$

Suppose that b) holds. Then \eqref{equation: restriction of scalars} implies that $C\sim C_\psi$ or $C\sim C_{\psi\varphi}$ and $D\sim D_\psi$ or $D\sim D_{\psi\varphi}$. We claim that there exists $i\in \{0,1\}$ such that $C\sim C_{\psi\varphi^ i}$ and $D\sim D_{\psi\varphi^ i}$. Indeed, otherwise $A$ would be the product of three elliptic curves with CM by the quadratic fields attached to the characters $\varphi,\psi,\varphi\psi\not =1$. But this is absurd, for the quadratic fields attached to these characters cannot all be simultaneously imaginary. By the claim, we have
$$
A'\sim E\times C_\psi\times D_\psi \sim E\times C_{\varphi^i}\times D_{\varphi^i}\sim A_{\varphi^i}\,. 
$$

Suppose finally that c) holds. In this case, \eqref{equation: restriction of scalars} amounts to
$$
(1+\xi)(1+\varphi)\Tr(\varrho_{C,\ell})=\psi(1+\xi)(1+\varphi)\Tr(\varrho_{C,\ell})\,.
$$
Since there exists a density 1 subset $\Sigma'$ of $\Sigma_{kM}$ such that for every $\p\in \Sigma'$ the trace of $\varrho_{C,\ell}(\Frob_\p)$ is nonzero, the above equality implies 
$$
1+\xi|_{kM}=\psi|_{kM}(1+\xi|_{kM})\,.
$$
Hence, either $\psi|_{kM}=1$ or $\psi|_{kM}=\xi|_{kM}$. We deduce that there exist $i,j\in \{0,1\}$ such that $\psi=\varphi^i\xi^j$, and then $A'\sim E\times C_\psi\times C_{\xi\psi}\sim E\times C_{\varphi^i}\times C_{\xi\varphi^i}\sim A_{\varphi^i}$.
\end{proof}

\begin{corollary}\label{corollary: CM factor}
Let $B$ and $B'$ be abelian surfaces defined over $k$, and $E$ and $E'$ be elliptic curves defined over $k$ with CM. If $\Hom(E_\Qbar,B_{\Qbar})=0$ and \eqref{equation: main condition} holds for $A:=E\times B$ and $A':=E'\times B'$, then $A$ and $A'$ are quadratic twists. 
\end{corollary}

\begin{proof}
By Theorem \ref{theorem: serre}, there exists a quadratic character $\varphi$ of $G_{kM}$ such that
$$
E'_{kM}\times B'_{kM}\sim E_{kM,\varphi}\times B_{kM,\varphi}\,.
$$
Since $\Hom(E_\Qbar,B_{\Qbar})=0$, it follows that $E'_{kM}\sim E_{kM,\varphi}$ and $B'_{kM}\sim B_{kM,\varphi}$. In particular, for every prime $\p\in \Sigma_k$ split in $kM$, except for a density~$0$ set,~$E$ and~$E'$ (resp.~$B$ and~$B'$) are locally quadratic twists at $\p$. Similarly, for every $\p\in \Sigma_k$ inert in $kM$, except for a density 0 set, the facts that $L_\p(E,T)=1+\Nm(\p)T^2$ and that~$A$ and~$A'$ are locally quadratic twists at $\p$ imply that $E$ and $E'$ (resp.~$B$ and~$B'$) are locally quadratic twists at~$\p$. Hence, by Corollary \ref{corollary: ellipticcurves} and the case $g=2$ of Theorem \ref{theorem: Main}, there exist quadratic characters $\chi,\psi$ of~$G_k$ such that $E'\sim E_\chi$ and $B'\sim B_\psi$, and then the corollary follows from Proposition~\ref{proposition: keycaseL}. 
\end{proof}

\begin{corollary}
If $A$ has absolute type \Fb, \Jb, or \Lb, then Theorem \ref{theorem: Main} holds.
\end{corollary}

\begin{proof}
It suffices to note that by Lemma \ref{lemma: splittingoverk} we can apply Corollary \ref{corollary: CM factor}.
\end{proof}

\subsubsection{Absolute type \Bb} In this case $\End^0(A_\Qbar)$ is a quadratic imaginary field. If $K=k$, then Theorem \ref{theorem: Main} follows from Theorem \ref{theorem: serre}. Otherwise, $K/k$ is quadratic and $A$ and $A'$ have Sato--Tate group $\stgroup[N(\Unitary(3))]{1.6.B.2.1a}$. Then, for a prime $\ell$ totally split in $M$, we have that $V_\ell(A)\simeq \Ind_K^k(V_\lambda(A))$, where~$\lambda$ is an embedding of $M$ into $\Q_\ell$. Theorem \ref{theorem: Main} then follows from the case $K=k$ and Lemma \ref{lemma: greatly simplifying lemma}.

\subsubsection{Absolute type \Eb} We will need the following lemma.

\begin{lemma}\label{lemma: simplifying lemma degree 3}
Suppose that \eqref{equation: main condition} holds for $A$ and $A'$, and that:
\begin{enumerate}[i)]
\item The endomorphism field $K/k$ of $A$ is a degree $3$ cyclic extension. 
\item $A$ has Frobenius traces concentrated in $K$.
\end{enumerate} 
Then $A$ and $A'$ are quadratic twists.
\end{lemma}

\begin{proof}
By Theorem \ref{theorem: mild refinement serre}, there is a quadratic character $\chi$ of $G_K$ such that~$A'_{K}\sim A_{K,\chi}$ cutting out an extension $L/K$ such that $L/k$ is Galois. If $L=K$, the lemma is clear. Otherwise $\Gal(L/k)\simeq \sym 3$ of $\cyc 6$. In any case, there is a quadratic extension $F/k$ such that $L=FK$. Let $\psi$ be the nontrivial quadratic character of $F/k$, so that $\psi|_K=\chi$. By $ii)$, we have $\varrho_{A,\ell}\otimes \psi = \varrho_{A,\ell}\otimes \chi$, and the lemma follows. 
\end{proof}

Let us return to the case that $A$ has absolute type $\Eb$. We can choose a prime~$\ell$ such that $\varrho_{A,\ell}|_K\simeq \varrho_1\oplus \varrho_2 \oplus \varrho_3$,
where the $\varrho_i\colon G_K\rightarrow \GL_2(\Q_\ell)$ are strongly absolutely irreducible pairwise nonisomorphic representations. If $\varrho_{A,\ell}$ is reducible, then one of the $\varrho_i$ descends to $G_k$ and the Zariski closure of its image is connected. Theorem~\ref{theorem: Main} then follows from Theorem~\ref{theorem: serre}.
Suppose that $\varrho_{A,\ell}$ is irreducible. Then, $K/k$ has degree~$3$ or ~$6$, depending on whether the Sato--Tate group of $A$ is $\stgroup[E_s]{1.6.E.3.1a}$ or $\stgroup[E_{s,t}]{1.6.E.6.1a}$. In the first case, Theorem \ref{theorem: Main} follows from Lemma \ref{lemma: simplifying lemma degree 3}. In the second case, let $N/k$ be the quadratic subextension of $K/k$. By the previous case, $A_N$ and $A_N'$ are quadratic twists, and then Theorem \ref{theorem: Main} follows from Lemma \ref{lemma: greatly simplifying lemma}.

\subsubsection{Absolute type \Hb} We may assume that $A$ is absolutely simple and that $\End^0(A_\Qbar)$ is a sextic CM field $M$, as otherwise Theorem \ref{theorem: Main} follows from Corollary~\ref{corollary: CM factor}. 
From \cite[\S4.3]{FKS21c}, we see that one of the following cases occurs:
\begin{enumerate}[i)]
\item $K=k$ and $\End^0(A)\simeq M$.
\item $K/k$ is quadratic and $\End^0(A)$ is a real cubic field.
\item $K/k$ is cyclic of order $3$ and $\End^0(A)$ is an imaginary quadratic field. 
\item $K/k$ is cyclic of order $6$ and $\End^0(A)\simeq \Q$. 
\end{enumerate}
Let $\ell$ be a prime totally split in $M$. If we set $n=6/[K:k]=[\End^0(A):\Q]$, then $V_\ell(A)$ admits a decomposition analogous to \eqref{equation: moduledescription}, and hence $A$ has Frobenius traces concentrated in $K$. In case i), Theorem \ref{theorem: Main} follows from Theorem \ref{theorem: serre}. In case ii) (resp. iii), iv)), it follows from case i) and Lemma \ref{lemma: greatly simplifying lemma} (resp. case i) and Lemma~\ref{lemma: simplifying lemma degree 3}, case iii) and Lemma~\ref{lemma: greatly simplifying lemma}).

\section{Examples}\label{section: examples}

In this final section, we provide two examples. Let $A$ be an abelian variety defined over the number field $k$, and let $\ell$ be a prime. For a prime $\p$ of $\Sigma_k$ coprime to $\ell$, we let 
$a_\p(A)$ denote $\Tr(\varrho_{A,\ell}(\Frob_\p))$. 
\subsection{First example} We will exhibit two abelian surfaces $A$ and $A'$ defined over $\Q$ which, despite not being quadratic twists, satisfy $a_p(A)=\pm a_p(A')$ for all rational odd primes $p$. Let $C$ and $C'$ denote the genus $2$ curves defined over~$\Q$ given by 
$$
y^2=x^5-x \qquad{\text{and}} \qquad y^2=x^5+4x\,.
$$
Let $A$ and $A'$ denote the Jacobians of $C$ and $C'$. Note that $C$ and $C'$ have good reduction outside $2$ and thus so do $A$ and $A'$. 

\begin{proposition}
For every odd prime $p$, we have $a_p(A)=\pm a_p(A')$. Nonetheless,~$A$ and~$A'$ are not quadratic twists.
\end{proposition}

\begin{proof}
By computing $L_3(A,T)$ and $L_3(A',T)$, one sees that $A_3$ and $A'_3$ are not quadratic twists, and hence neither are $A$ and $A'$. Let $p$ be an odd prime. We will rely on the results of \cite{FS14} in order to show that $a_p(A)=\pm a_p(A')$. Accordingly to \cite[Table 5]{FS14}, define $K=\Q(\sqrt 2,i)$, $L=K(2^{1/4})$, and $L'=K(\sqrt{2+\sqrt 2})$. Let~$r$ (resp. $s$, $s'$) denote the residue degree of $p$ in $K$ (resp. $L$, $L'$). Since $L\cap L'=K$ and $\Gal(L/\Q)\simeq \cyc 2\times \cyc 4$ and $\Gal(L/\Q)\simeq \dih 4$, we have three cases: i) if $r=1$, then $s,s'=1$ or $2$; if $r=2$, then $s,s'=2$ or $4$; if $r=4$, then $s=s'=4$.
Then \cite[Prop. 4.9]{FS14}, implies that $a_p(A)=a_p(A')=0$ in cases ii) and iii), and that $a_p(A)=\pm a_p(A')$ in case i).  
\end{proof} 
 
\subsection{Second example}
We will exhibit two abelian fourfolds $A$ and $A'$ defined over $\Q$ which, despite not being quadratic twists, are locally quadratic twists at all rational odd primes. This example was obtained by means of a computer exploration of the family of curves $y^2=x^9+ax$, with $a\in \Q$, carried by Edgar Costa, to whom I express my deepest gratitude.
Let $C$ and $C'$ denote the genus $4$ curves defined over $\Q$ and given by the affine models
$$
y^2=x^9+x\qquad \text{and} \qquad y^2=x^9+16x\,.
$$ 
Let $A$ and $A'$ denote the Jacobians of $C$ and $C'$. Note that $C$ and $C'$ have good reduction outside $2$ and thus so do $A$ and $A'$. 

\begin{proposition}
For every odd prime $p$, the reductions $A_p$ and $A'_p$ are quadratic twists. Nonetheless, $A$ and $A'$ are not.
\end{proposition}

\begin{proof}
Let $L/\Q$ denote the minimal extension over which all homomorphisms from $A_\Qbar$ to $A'_\Qbar$ are defined. By \cite[Thm. 4.2]{Sil92}, the extension $L/\Q$ is finite, Galois, and unramified outside $2$. We first show that $A$ and $A'$ are not quadratic twists. Suppose the contrary, that is, that there exists a quadratic character $\chi\colon G_\Q\rightarrow \{\pm 1\}$ such that $A'\sim A_\chi$. Note that $\chi$ necessarily factors through an at most quadratic subextension $N/\Q$ of $L/\Q$. Hence, $N/\Q$ is unramified outside $2$, which means that $N$ is either $\Q$, $\Q(\sqrt 2)$, $\Q(i)$, or $\Q(\sqrt{-2})$. One easily computes
$$
\Tr\varrho_{A,\ell}(\Frob_{17})=-8\,,\qquad \Tr\varrho_{A',\ell}(\Frob_{17})=8\,,
$$
which implies that $17$ is inert in $N$. This is however a contradiction with the possibilities determined for $N$.

Let $p$ be an odd prime. We will next show that $A_p$ and $A_p'$ are quadratic twists. Let $F$ denote $\Q(2^{1/4})$. Note that the map
\begin{equation}\label{equation: isomorphism214}
\phi\colon C_{F}\rightarrow C'_F\,,\qquad \phi(x,y)=(\sqrt 2x,2^{9/4}y)
\end{equation}
defines an isomorphism. If $p\equiv 1,7 \pmod 8$, then $a:=2^{1/2}\in \F_p$. The isomorphism
 shows that~$A_p$ and~$A'_p$ are quadratic twists. Suppose from now on that $p \equiv 3,5\pmod 8$. We claim that $A_p$ and $A'_p$ are in fact isogenous. By Lemma~\ref{lemma: Nie} below, there exist integers $s$ and $s'$ such that
\begin{equation}\label{equation: Lpolyshape}
L_p(A,T)=1+sT^4+p^4T^8\,,\quad L_p(A',T)=1+s'T^4+p^4T^8\,.
\end{equation}
The relations between elementary symmetric polynomials and power sums show that \eqref{equation: Lpolyshape} implies that $\#C(\F_{p^i})=\#C'(\F_{p^i})=p^i+1$ for $i=1,\dots, 3$. But \eqref{equation: isomorphism214} immediately implies that $\#C(\F_{p^4})=\#C'(\F_{p^4})$. Thus $L_p(A,T)= L_p(A',T)$, and ~$A_p$ and~$A'_p$ are isogenous.
\end{proof}

\begin{lemma}\label{lemma: Nie}
Let $A$ denote the Jacobian of a curve defined by an affine model $y^2=x^9+cx$ for $c\in \Q^\times$. For every prime $p\equiv 3,5\pmod 8$ of good reduction for~$A$, there exists an integer $s$ such that the $L$-polynomial of $A$ at $p$ is of the form
$$
L_p(A,T)=1+s T^4 +p^4T^8\,.
$$
\end{lemma}

\begin{proof}
We will apply \cite[Thm. 1.1]{Nie16} with $m_1=9$, $n_1=0$, $m=1$, $n=2$, $k_1=c$, $k_2=-1$, $\xi_1=7/16$, and $\xi_2=1/2$. In light of part (2) of loc. cit., it suffices to show that the integer $\mu(\xi)$ appearing in that formula is $4$. Here $\xi$ denotes the pair of rational numbers $(\xi_1,\xi_2)$.  By the discussion preceding \cite[Thm. 1.1]{Nie16}, $\mu(\xi)$ is the order of $p$ in $(\Z/16\Z)^\times$. If $p\equiv 3,5\pmod 8$, then this order is $4$.
\end{proof}


\begin{thebibliography}{McK-Sta}
\bibitem[BK15]{BK15} G. Banaszak and K.S. Kedlaya, \emph{An algebraic Sato–Tate group and Sato--Tate conjecture}, Indiana Univ. Math. J. \textbf{64} (2015), 245--274.

\bibitem[Fal83]{Fal83} G. Faltings, \emph{Endlichkeitss\"atze f\"ur abelsche Variet\"aten \"uber Zahlk\"orpern}, Invent. Math. \textbf{73} (1983), 349--366.

\bibitem[Fit12]{Fit12} F. Fit\'e, \emph{Artin representations attached to pairs of isogenous abelian varieties}, J. Number Theory \textbf{133}:4 (2013), 1331--1345.

\bibitem[FG18]{FG18} F. Fit\'e, X. Guitart, \emph{Fields of definition of elliptic $k$-curves and the realizability of all genus $2$ Sato-Tate groups over a number field}, Trans. Amer. Math. Soc. \textbf{370}, n. 7 (2018), 4623--4659.

\bibitem[FG22]{FG20} F. Fit\'e, X. Guitart, \emph{Tate module tensor decompositions and the Sato--Tate conjecture for certain abelian varieties potentially of $\GL_2$-type}, Mathematische Zeitschrift \textbf{300} (2022), 2975--2995.

\bibitem[FKRS12]{FKRS12} F. Fit\'e, K.S. Kedlaya, V. Rotger, and A.V. Sutherland, \emph{Sato--Tate distributions and Galois endomorphism modules in genus 2}, Compositio Mathematica \textbf{148}, n. 5 (2012), 1390--1442. 

\bibitem[FKS21a]{FKS21a} F. Fit\'e, K.S. Kedlaya, A.V. Sutherland, \emph{Sato--Tate groups of abelian threefolds: a preview of the classification}, Contemp. Math. \textbf{770} (2021), 103--129.

\bibitem[FKS21b]{FKS21c} F. Fit\'e, K.S. Kedlaya, A.V. Sutherland, \emph{Sato--Tate groups of abelian threefolds}, \arXiv{2106.13759}{2} (2021).

\bibitem[FS14]{FS14}
F. Fit\'e and A.V. Sutherland, \emph{Sato--Tate distributions of twists of $y^2=x^5-x$ and $y^2=x^6+1$},
\textit{Algebra Number Theory} \textbf{8} (2014), 543--585.

\bibitem[Hec20]{Hec20} E. Hecke, \emph{Eine neue Art von Zetafunktionen und ihre Beziehungen zur Verteilung der Primzahlen}, Math. Z. \textbf{6} (1920), no. 1-2, 11--51,

\bibitem[Kat81]{Kat81} N.M. Katz, \emph{Galois properties of torsion points on abelian varieties}, Invent. Math. \textbf{62}, 481--502 (1981).

\bibitem[KL20]{KL20} C.B. Khare, M. Larsen, \emph{Abelian varieties with isogenous reductions}, C. R. Math. Acad. Sci. Paris \textbf{358} (2020), no. 9-10, 1085--1089.

\bibitem[Kin98]{Kin98} G. Kings \emph{Higher regulators, Hilbert modular surfaces, and special values of $L$-functions}, Duke Math. J. \textbf{92}, no. 1 (1998), 61--127.

\bibitem[LMFDB]{LMFDB}
The LMFDB Collaboration, \textit{The L-functions and Modular Forms Database}, release 1.2.1, \url{https://www.lmfdb.org}, 2021
(accessed July 2021).

\bibitem[Mil72]{Mil72} J. Milne, \emph{On the arithmetic of abelian varieties}, Invent. Math. \textbf{17}, 177--190 (1972).

\bibitem[Nie16]{Nie16} M. Nie, \emph{Zeta functions of trinomial curves and maximal curves}, Finite Fields and Their Applications \textbf{39} (2016) 52--82. 

\bibitem[Raj98]{Raj98} C.S. Rajan, \emph{On strong multiplicity one for $\ell$-adic representations}. Internat. Math. Res. Notices (3) 1998, 161--172.

\bibitem[Ram00]{Ram00} D. Ramakrishnan, \emph{Recovering modular forms from squares}, appendix to ``A problem of Linnik
for elliptic curves and mean-value estimates for automorphic representations” by
W. Duke and E. Kowalski, Invent. Math. \textbf{139} (2000), 1--39.

\bibitem[Rib76]{Rib76} K.A. Ribet, \emph{Galois action on division points on abelian varieties with many real multiplications}, Am. J. Math. \textbf{98} (1976), 751--804.

\bibitem[Saw16]{Saw16} W. Sawin, \emph{Ordinary primes for Abelian surfaces}. C. R. Math. Acad. Sci. Paris \textbf{354} (2016), no. 6, 566--568.

\bibitem[Ser72]{Ser72} J.-P. Serre, \emph{Propriet\'es galoisennes des points d'ordre fini des courbes elliptiques}, Invent. Math. \textbf{15} (1972), 259--331.

\bibitem[Ser81]{Ser81} J.-P. Serre, \emph{Quelques applications du th\'eor\`eme de densit\'e de Chebotarev}, Inst. Hautes \'Etudes Sci. Publ. Math. No. \textbf{54} (1981), 323--401.

\bibitem[Ser12]{Ser12} J.-P. Serre, \emph{Lectures on $N_X(p)$}, CRC Press, Boca Raton, FL, 2012.

\bibitem[Shi71]{Shi71} G. Shimura, \emph{On the zeta-function of an abelian variety with complex multiplication}, Ann. of Math. (2) \textbf{94} (1971), 504--533.

\bibitem[Shi98]{Shi98} G. Shimura, \emph{Abelian varieties with complex multiplication and modular forms}, Princeton University Press, Princeton, NJ, 1998.

\bibitem[Sil92]{Sil92} A. Silverberg, \emph{Fields of definition for homomorphisms of abelian varieties}, J. Pure Appl. Algebra \textbf{77} (1992), 253--262.

\bibitem[Won99]{Won99} S. Wong, \emph{Twists of Galois representations and projective automorphisms}, Journal of Number Theory \text{74}, 1--18 (1999).
\end{thebibliography}
\end{document}